\documentclass[onefignum,onetabnum,nohypdvips]{siamart251216_latexmlpatched}  

\usepackage[utf8]{inputenc}

\usepackage{latexml}
\iflatexml\providecommand{\headers}[2]{}
\providecommand{\funding}[1]{\\\textbf{Funding:} #1}
\providecommand{\newsiamremark}[2]{}
\providecommand{\apptocmd}[4]{}
\newenvironment{MSCcodes}{\textbf{MSC codes:}}{}
\newenvironment{DOI}{\textbf{DOI:}}{}
\fi

\usepackage{tikz}
\usetikzlibrary{positioning}

\usepackage{amsmath}
\usepackage{bm}  
\usepackage{mathrsfs}  
\usepackage{amssymb}  

\makeatletter{}
\let\orig@newsiamremark\newsiamremark{}
\renewcommand{\newsiamremark}[2]{%
  \orig@newsiamremark{#1}{#2}%
  \AddToHook{env/#1/begin}{\crefalias{theorem}{#1}}%
}%
\makeatother{}
\newsiamremark{remark}{Remark}
\newsiamremark{example}{Example}

\usepackage{mathtools}

\crefname{lemma}{Lemma}{Lemmata}

\newcommand{\bbone}{\boldsymbol{1}}

\newcommand{\C}{\mathbb{C}}
\newcommand{\R}{\mathbb{R}}

\newcommand{\V}{\mathrm{V}}

\newcommand{\D}{\mathrm{D}}

\DeclareMathOperator{\Div}{Div}
\DeclareMathOperator{\DIV}{DIV}
\renewcommand{\div}{\operatorname{div}}

\newcommand{\dev}{\operatorname{dev}}
\newcommand{\sym}{\operatorname{sym}}
\newcommand{\stf}{\operatorname{stf}}
\DeclareMathOperator{\Stf}{Stf}
\DeclareMathOperator{\Dev}{Dev}
\DeclareMathOperator{\Sym}{Sym}

\DeclareMathOperator{\tr}{tr}

\newcommand{\lebe}{\operatorname{L}}
\newcommand{\sobo}{\operatorname{W}}
\newcommand{\hil}{\operatorname{H}}

\newcommand{\norm}[1]{\left\lVert#1\right\rVert}
\newcommand{\skalarProd}[2]{#1\boldsymbol{\cdot}#2}

\newcommand{\abs}[1]{\lvert#1\rvert}

\headers{R13 Well-Posedness via Tensor Korn Inequalities}{Peter Lewintan, Lambert Theisen, and Manuel Torrilhon}
\title{
  Well-Posedness of the Linear Regularized 13-Moment Equations Using Tensor-Valued Korn Inequalities%
  \thanks{%
  This preprint reflects the status as of \today{} and is subject to change.
  Modifications resulting from the publishing process, such as peer review, editing, and formatting, may not be incorporated here. Please refer to the final published version for the most accurate content.%
  \funding{Supported by the German Research Foundation (DFG), project 442047500 (LT/MT).}%
  }%
}

\author{
Peter Lewintan\thanks{%
  Institute for Analysis, Karlsruhe Institute of Technology, Englerstr.\ 2, 76131 Karlsruhe, Germany (\email{peter.lewintan@kit.edu})
}
\and
Lambert Theisen\thanks{%
  Corresponding author. Applied and Computational Mathematics, RWTH Aachen University, Schinkelstr.\ 2, 52062 Aachen, Germany (\email{lambert.theisen@rwth-aachen.de}, \url{https://thsn.dev})
}
\and
Manuel Torrilhon\thanks{%
  Applied and Computational Mathematics, RWTH Aachen University, Schinkelstr.\ 2, 52062 Aachen, Germany (\email{mt@acom.rwth-aachen.de})
}
}

\ifpdf{}
\hypersetup{
  pdftitle={Well-Posedness of the R13 Equations Using Tensor-Valued Korn Inequalities},
  pdfauthor={Peter Lewintan, Lambert Theisen, Manuel Torrilhon}
}
\fi

\begin{document}

\maketitle

\begin{abstract}
  In this paper, we finally prove the well-posedness of the linearized R13 moment model, which describes, e.g., rarefied gas flows. As an extension of the classical fluid equations, moment models are robust and have been frequently used, yet they are challenging to analyze due to their additional equations. By effectively grouping variables, we identify a 2-by-2 block structure, allowing us to analyze well-posedness within the abstract LBB framework for saddle point problems. Due to the unique tensorial structure of the equations, in addition to an interesting combination of tools from Stokes' and linear elasticity theory, we also need new coercivity estimates for tensor fields. These Korn-type inequalities are established by analyzing the symbol map of the symmetric and trace-free part of tensor derivative fields. Together with the corresponding right inverse of the tensorial divergence, we obtain the existence and uniqueness of weak solutions. This result also serves as the basis for future numerical analysis of corresponding discretization schemes.
\end{abstract}



\begin{DOI}
  Preprint
\end{DOI}


\section{Introduction}\label{s_intro}
In this paper, we propose and analyze a mixed formulation for the linear regularized 13-moment equations (R13)~\cite{struchtrupRegularizationGrads132003} on bounded Lipschitz domains \(\Omega \subset \mathbb{R}^3\) with boundary \(\Gamma \coloneqq \partial \Omega\), i.e., $\Omega$ is an open, non-empty, connected set whose boundary $\Gamma$ can be locally expressed as the graph of a Lipschitz continuous function. The R13 equations are a system of partial differential equations that describe the evolution of the macroscopic quantities in, e.g., rarefied gases or microfluids. In these flows, the length of the mean free path between particle collisions becomes significant when compared to a reference length scale of the process. The ratio between these length scales is typically given by the \emph{Knudsen number} \(\operatorname{Kn} >0\), which serves as a parameter in the equations. Let \(m:\Omega \to \mathbb{R}\), \(\boldsymbol{b}:\Omega \to \mathbb{R}^3\), and \(r:\Omega \to \mathbb{R}\) be a mass source, a body force, and an energy source. Similar to the Stokes equations combined with Fourier's law of heat conduction, we are interested in the fields for the pressure \(p : \Omega \to \mathbb{R}\), the velocity \(\boldsymbol{u} : \Omega \to \mathbb{R}^3\), and the temperature \(\theta : \Omega \to \mathbb{R}\), given by
\begin{equation}\label{eq_balance1}
  \left\{
  \begin{aligned}
    \operatorname{div}\boldsymbol{u} &= m, \\
    \boldsymbol{\nabla} p + \operatorname{Div}\boldsymbol{\sigma} &= \boldsymbol{b}, \\
    \operatorname{div}\boldsymbol{u} + \operatorname{div}\boldsymbol{s} &= r,
  \end{aligned}
  \quad \textnormal{in } \Omega,
  \right.
\end{equation}
where we distinguish between \(\operatorname{div} \boldsymbol{u}: \Omega \to \mathbb{R}\), \(\operatorname{div} \boldsymbol{u} \coloneqq \sum_j \partial_j u_j\), and \(\operatorname{Div} \boldsymbol{\sigma}: \Omega \to \mathbb{R}^3, (\operatorname{Div} \boldsymbol{\sigma})_i \coloneqq \sum_j \partial_j \sigma_{ij}\), acting on vectors and matrices (row-wise), respectively. A detailed explanation of the notation is provided in \cref{subsec:notation}. Instead of using closure relations for the heat flux vector \(\boldsymbol{s} : \Omega \to \mathbb{R}^3\) and the symmetric and trace-free stress tensor \(\boldsymbol{\sigma} : \Omega \to \mathbb{R}^{3 \times 3}_{\operatorname{stf}} \coloneqq \{\boldsymbol{A} \in \mathbb{R}^{3 \times 3} : \boldsymbol{A}=\boldsymbol{A}^\top, \operatorname{tr}\boldsymbol{A}=0\}\) in \cref{eq_balance1}, we solve their corresponding evolution equations given by 
\begin{equation}\label{eq_balance2}
  \left\{
  \begin{aligned}
    \tfrac{4}{5} \operatorname{stf}\textnormal{D} \boldsymbol{s} + 2 \operatorname{stf} \textnormal{D} \boldsymbol{u}
     + \operatorname{DIV}\boldsymbol{m} &= - \tfrac{1}{\operatorname{Kn}} \boldsymbol{\sigma},
    \\
    \tfrac{5}{2} \boldsymbol{\nabla} \theta + \operatorname{Div} \boldsymbol{\sigma} + \tfrac{1}{2} \operatorname{Div} \boldsymbol{R} + \tfrac{1}{6} \boldsymbol{\nabla} \triangle &= - \tfrac{1}{\operatorname{Kn}} \tfrac{2}{3} \boldsymbol{s},
  \end{aligned}
  \quad \textnormal{in } \Omega.
  \right.
\end{equation}
In \cref{eq_balance2}, \(\boldsymbol{m} : \Omega \to \mathbb{R}^{3 \times 3 \times 3}_{\operatorname{Stf}}\), \(\boldsymbol{R}: \Omega \to \mathbb{R}^{3 \times 3}_{\operatorname{stf}}\), and \(\triangle : \Omega \to \mathbb{R}\) are the \emph{highest-order moments}; \((\operatorname{DIV} \boldsymbol{m})_{ij} \coloneqq \sum_{k} \partial_k m_{ijk}\) defines the matrix field \(\operatorname{DIV} \boldsymbol{m}: \Omega \to \mathbb{R}^{3 \times 3}\); and \(\operatorname{stf}: \mathbb{R}^{3 \times 3} \to \mathbb{R}^{3 \times 3}_{\operatorname{stf}}\) with \(\operatorname{stf}\boldsymbol{A} \coloneqq \tfrac{1}{2}(\boldsymbol{A}+\boldsymbol{A}^\top) - \tfrac{\operatorname{tr} \boldsymbol{A}}{3} \bbone_3\) returns the symmetric, trace-free part of a matrix.  
For the highest-order moments, we use the regularized closure relations given by
\begin{equation}\label{eq_closures}
  \boldsymbol{m} = - 2 \operatorname{Kn} \operatorname{Stf} \textnormal{D} \boldsymbol{\sigma}
  ,\quad
  \boldsymbol{R} = - \tfrac{24}{5} \operatorname{Kn} \operatorname{stf}\textnormal{D} \boldsymbol{s}
  ,\quad
  \triangle = - 12 \operatorname{Kn} \operatorname{div} \boldsymbol{s}
  ,
\end{equation}
where \(\textnormal{D} \boldsymbol{\sigma}: \Omega \to \mathbb{R}^{3 \times 3 \times 3}\) is defined by \((\textnormal{D} \boldsymbol{\sigma})_{ijk} \coloneqq \partial_k \sigma_{ij}\), and \(\operatorname{Stf}: \mathbb{R}^{3 \times 3 \times 3} \to \mathbb{R}^{3 \times 3 \times 3}_{\operatorname{Stf}}\) by \((\operatorname{Stf} \boldsymbol{T})_{ijk} \coloneqq T_{(ijk)} - \tfrac{1}{5}(T_{(ill)} \delta_{jk} + T_{(ljl)} \delta_{ik} + T_{(llk)} \delta_{ij})\)~\cite{struchtrupMacroscopicTransportEquations2005}, where, with \(S_3\) denoting the symmetric group, \(T_{(ijk)} \coloneqq \frac{1}{6} \sum_{\pi \in S_3} T_{\pi(i)\pi(j)\pi(k)}\) defines the symmetric part of \(\boldsymbol{T}\). The coefficients in \cref{eq_balance2,eq_closures}, which may appear unconventional at first glance, result from the assumption of so-called \emph{Maxwell molecules}\footnote{See, e.g.,~\cite{struchtrupRegularized13Moment2013} for other molecular potentials that yield different coefficients.} in the underlying kinetic derivation (see the next \cref{subsec:sota}) and are, therefore, part of the model.
\par
With the boundary-aligned moving $3$-frame \((\boldsymbol{t}_1,\boldsymbol{t}_2,\boldsymbol{n})\) (i.e., unit outer normal \(\boldsymbol{n} \in \mathbb{R}^3\) and \(\boldsymbol{t}_1, \boldsymbol{t}_2 \in \mathbb{R}^3\) spanning the tangential plane), let \(u_n \coloneqq \sum_{i} u_i n_i\), \(R_{nn} \coloneqq \sum_{i,j} R_{ij} n_i n_j\), \(m_{nnn} \coloneqq \sum_{i,j,k} m_{ijk} n_i n_j n_k\), and other components be defined analogously. Then, we use Onsager boundary conditions (BCs) (see~\cite{ranaThermodynamicallyAdmissibleBoundary2016} with the thermodynamic adaptation made in~\cite{torrilhonHierarchicalBoltzmannSimulations2017}) given by  
\begin{equation}\label{bcs_3d}
\begin{aligned}
  \left(u_n - u_n^{\mathrm{w}}\right) &= \epsilon^{\mathrm{w}} \tilde{\chi} \left( (p-p^{\mathrm{w}}) + \sigma_{nn} \right),
  \\
  \sigma_{nt_{i}} &= \tilde{\chi} \left( (u_{t_i}-u_{t_i}^{\mathrm{w}}) + \tfrac{1}{5} s_{t_i} + m_{nnt_{i}} \right) \quad (i \in \{1,2\}),
  \\
  R_{nt_{i}} &= \tilde{\chi} \left( -(u_{t_i}-u_{t_i}^{\mathrm{w}}) + \tfrac{11}{5} s_{t_i} - m_{nnt_{i}} \right) \quad (i \in \{1,2\}),
  \\
  s_n &= \tilde{\chi} \left( 2(\theta-\theta^{\mathrm{w}}) + \tfrac{1}{2} \sigma_{nn} + \tfrac{2}{5} R_{nn} + \tfrac{2}{15} \triangle \right),
  \\
  m_{nnn} &= \tilde{\chi} \left( - \tfrac{2}{5} (\theta-\theta^{\mathrm{w}}) + \tfrac{7}{5} \sigma_{nn} - \tfrac{2}{25} R_{nn} - \tfrac{2}{75} \triangle \right),
  \\
  \left( \tfrac{1}{2} m_{nnn} + m_{nt_1t_1} \right) &= \tilde{\chi} \left( \tfrac{1}{2} \sigma_{nn} + \sigma_{t_1t_1} \right),
  \\
  m_{nt_1t_2} &= \tilde{\chi} \sigma_{t_1t_2}
  ,
\end{aligned}
\end{equation}
in which \(u_n^{\mathrm{w}}, u_{t_1}^{\mathrm{w}}, u_{t_2}^{\mathrm{w}}, p^{\mathrm{w}}, \theta^{\mathrm{w}} \in \mathbb{R}\) are the velocity components, the pressure, and the temperature on \(\Gamma\), \(\epsilon^{\mathrm{w}} \ge 0\)\label{ieq:epsilonw} is a parameter controlling the velocity prescription strength, and \(\tilde{\chi} > 0\) is the modified accommodation factor~\cite{westerkampFiniteElementMethods2019}. The (linear) Onsager BCs~\cref{bcs_3d} are the moment-level counterpart~\cite{torrilhonBoundaryConditionsRegularized2008} of the \emph{Maxwell accommodation} boundary model from kinetic theory: they relate normal fluxes at the wall (e.g., \(u_n,\sigma_{nt_i},R_{nt_i},s_n,m_{nnn},\dots\)) to the corresponding thermodynamic driving forces (jumps between gas and wall state) through a symmetric, positive map, consistent with the \emph{Curie principle} (only quantities of equal tensorial rank are coupled). In this form, the boundary entropy production is non-negative, so the second law of thermodynamics is respected; see~\cite{ranaThermodynamicallyAdmissibleBoundary2016,struchtrupTheoremRegularizationBoundary2007}. While the concrete coefficients depend on the molecular and accommodation models, other thermodynamically admissible (possibly simplified) boundary closures could change the coefficients and/or the coupling structure.
\par
Eliminating the highest-order moments by inserting the closures \cref{eq_closures} into \cref{eq_balance1,eq_balance2}, while prescribing \cref{bcs_3d}, yields the R13 boundary-value problem, which we analyze in this work.  

\vspace{0.07cm}
\subsection{State of the art, context, and our contribution}\label{subsec:sota}
We start with an example that motivates the modeling of rarefied gases in bounded domains \textendash\ namely, thermal transpiration-based gas pumps which use a naturally occurring zeolite~\cite{GuptaGianGasPump} to generate a gas flow. The nanopores in this mineral enable free-molecular flow even at atmospheric pressure. Unlike in the continuum gas flow regime, in the free-molecular regime, the gas molecules bounce against the channel walls much more frequently than they bounce against each other. Under these conditions, wall interactions dominate, causing the molecules to drift from the cold end to the warm end; see~\cite{GuptaGianGasPump} for more details. The main advantage of this pump is that it has no moving parts and solely relies on temperature differences. Hence, it could provide reliable, precise control of gas flow for a variety of applications, such as gas-sensing breath analyzers or in satellite control; see also~\cite{SatelliteControl} for the latter. For further examples and applications of the so-called Knudsen pumps, we refer the reader to, e.g.,~\cite{KnudsenPump1,KnudsenPump2} and the references therein.
\par
At a high level, the system \cref{eq_balance1,eq_balance2,eq_closures,bcs_3d} results from multiple modeling and approximation steps. Starting from kinetic theory in the \(6\)-dimensional phase space, a Galerkin approximation to a lower-dimensional subspace yields a set of nonlinearly coupled, time-dependent moment equations~\cite{abdelmalikMomentClosureApproximations2016}. Using a Hermite basis~\cite{gradKineticTheoryRarefied1949} leads to such a model hierarchy (including, e.g., the R13, R26, and G45 models~\cite{ranaHtheoremBoundaryConditions2021,struchtrupRegularizationGrads132003,guHighorderMomentApproach2009}) that allows for systematic model improvement while still recovering the well-established Navier--Stokes and Fourier models in the limit \(\operatorname{Kn} \to 0\)~\cite{torrilhonModelingNonequilibriumGas2016,bungerStructuredDerivationMoment2023}. However, with an increased number of moments, the number of solution variables \textendash\ and at the same time \textendash\ the tensorial rank increases, and thus the complexity of the equations. Nevertheless, especially in the moderate Knudsen number regime, moment models have been shown to be an essential tool for simulating \textit{rarefaction effects}, i.e., phenomena that cannot be captured by the classical Navier--Stokes--Fourier equations~\cite{torrilhonModelingNonequilibriumGas2016}.
\par
The R13 model is a good compromise between accuracy and complexity~\cite{torrilhonModelingNonequilibriumGas2016}, and the inflow, outflow, and thermal boundary conditions enable the study of a wide range of applications, including microchannel flows~\cite{caiLinearRegularized13Moment2024,taheriCouettePoiseuilleMicroflows2009,linTimedependentRegularised13moment2025}, lid-driven cavities~\cite{ranaRobustNumericalMethod2013}, flow past an obstacle~\cite{westerkampFiniteElementMethods2019}, thermally induced flows~\cite{himanshiGeneralizedFundamentalSolution2026}, thermal edge flows~\cite{theisenFenicsR13TensorialMixed2021}, and phase transitions with moving interfaces~\cite{wenigerConformingInterfaceApproach2025}. However, exact solutions are known only for certain domains~\cite{martinAcousticScatteringRarefied2020,caiLinearRegularized13Moment2024,taheriCouettePoiseuilleMicroflows2009}. Numerical solutions, on the other hand, have been obtained, e.g., by using finite difference (FD) schemes~\cite{ranaRobustNumericalMethod2013}, finite volume (FV) methods~\cite{guComputationalStrategyRegularized2007}, discontinuous Galerkin (DG) approaches~\cite{torrilhonHierarchicalBoltzmannSimulations2017}, the method of fundamental solution (MFS)~\cite{lockerbyFundamentalSolutionsMoment2016}, and finite element methods (FEM)~\cite{westerkampContinuousInteriorPenalty2017,westerkampFiniteElementMethods2019,theisenFenicsR13TensorialMixed2021}. For the latter, the open-source solver developed in~\cite{theisenFenicsR13TensorialMixed2020} can simulate the model equations of this work, i.e., the steady-state, linearized, and dimensionless case in three dimensions on arbitrary geometries. However, even with these simplifications, the mathematical analysis of the R13 equations still lags, and even fundamental questions remain unclear: \emph{Do the R13 equations even have a weak solution?} If so, \emph{is this weak solution unique and continuously dependent on the data?} In this work, we answer both questions affirmatively, meaning that the equations are well-posed (see \cref{thm:mainR13}). The key idea is to consider the weak formulation within the abstract-saddle-point framework~\cite{brezziMixedHybridFinite1991} of mixed formulations. This well-posedness result is not only interesting in itself but also serves as the foundation for numerical efforts.
\par
The analysis requires new theoretical tools due to the equations' increased tensor rank. We combine estimates from the context of Stokes' equations, Poisson's equation, and, interestingly, linear elasticity. In particular, by applying the Nečas framework~\cite{necasNormesEquivalentesDans1966}, we will derive tensor-valued Korn inequalities that can be used for coercivity estimates in matrix equations (see our \cref{lem:CelliptSTF,lem:tensorKorn}). Indeed, Korn's inequalities are the key tool in (linear) elasticity and fluid mechanics, since they provide the necessary coercivity of the arising bilinear forms and yield well-posedness and regularity results; cf.~\cite{kornUberEinigeUngleichungen1909, horganKornsInequalitiesTheir1995, ciarletKornsInequality2010,duvautInequationsMecaniquePhysique1972} for an incomplete list and the discussion in the \cref{sec:Korn} below.
Korn inequalities have been generalized to many different contexts, e.g., to the geometrically nonlinear counterpart~\cite{frieseckeTheoremGeometricRigidity2002}, to the case of non-constant coefficients~\cite{neffKornsFirstInequality2002,pompeKornsFirstInequality2003}, or to the case of incompatible tensor fields~\cite{lewintanKornInequalitiesIncompatible2021,gmeinederKornMaxwellSobolevInequalities2021,gmeinederKornMaxwellSobolev2024,gmeinederLimitingKornMaxwellSobolevInequalities2024}; cf.\ also with the references in~\cite{lewintanKornInequalitiesIncompatible2021, gmeinederKornMaxwellSobolev2024}.
In particular, Korn--Maxwell--Sobolev, i.e., Korn inequalities for incompatible tensor fields, have been derived in the framework of gradient plasticity with plastic spin~\cite{garroniGradientTheoryPlasticity2010} or in extended continuum-type models like the relaxed micromorphic model~\cite{neffUnifyingPerspectiveRelaxed2014}, again to prove coercivity-type estimates.

\subsection{The connection to Stokes' and Poisson's problem}\label{subsec:stokes_poisson_connection}
Since the R13 equations \cref{eq_balance1,eq_balance2,eq_closures} appear quite complicated, we would like to explain their structure in further detail. In fact, the R13 equations are an extension to the classical Stokes and Poisson problems coupled together with additional moment terms that are higher-order in the Knudsen number \(\operatorname{Kn}\).
To see that, we set \(m=0\) in~\cref{eq_balance1}, such that the velocity field \(\boldsymbol{u}\) is divergence-free, and the heat equation simplifies to \(\operatorname{div} \boldsymbol{s} = r\). A multiplication of the two equations in \cref{eq_balance2} by \(\operatorname{Kn}\), while using that the all higher-order moments from \cref{eq_closures} are \(\mathcal{O}(\operatorname{Kn})\), yields \(\boldsymbol{\sigma} = - \operatorname{Kn} (2 \operatorname{stf} \textnormal{D} \boldsymbol{u} + \tfrac{4}{5} \operatorname{stf} \textnormal{D} \boldsymbol{s} + \mathcal{O}(\operatorname{Kn}))\) and \(\boldsymbol{s} = - \tfrac{3}{2}\operatorname{Kn} (\tfrac{5}{2} \boldsymbol{\nabla} \theta + \operatorname{Div} \boldsymbol{\sigma} + \mathcal{O}(\operatorname{Kn}))\). These two relations, in turn, imply that \(\boldsymbol{\sigma}\) and \(\boldsymbol{s}\), and therefore, also \(\operatorname{Div} \boldsymbol{\sigma}\) and \(\operatorname{stf} \textnormal{D} \boldsymbol{s}\) are of order \(\mathcal{O}(\operatorname{Kn})\), such that the coupling between \(\boldsymbol{\sigma}\) and \(\boldsymbol{s}\) is of order \(\mathcal{O}(\operatorname{Kn}^2)\) in total. Thus, the R13 equations further simplify to
\begin{equation}\label{eq_stokes_and_poisson_first_order}
  \left\{
  \begin{aligned}
    \boldsymbol{\nabla} p + \operatorname{Div} \boldsymbol{\sigma} &= \boldsymbol{b}
    , \\
    \operatorname{div} \boldsymbol{u} &= 0
    , \\
    \mathcal{O}(\operatorname{Kn}^2) + \boldsymbol{\sigma} &= -2 \operatorname{Kn} \operatorname{stf} \textnormal{D} \boldsymbol{u}
    ,
  \end{aligned}
  \right.
  \quad
  \left\{
    \begin{aligned}
      \operatorname{div} \boldsymbol{s} &= r
      , \\
      \mathcal{O}(\operatorname{Kn}^2) + \boldsymbol{s} &= - \tfrac{15}{4} \operatorname{Kn} \boldsymbol{\nabla} \theta
      .
    \end{aligned}
  \right.
\end{equation}
After neglecting the \(\mathcal{O}\left(\operatorname{Kn}^2\right)\)-terms in \cref{eq_stokes_and_poisson_first_order}, we obtain the Stokes problem in first-order form with the Navier--Stokes law as closure for \(\boldsymbol{\sigma}\) as well as the Poisson (or Darcy) problem with Fourier's law as closure for \(\boldsymbol{s}\). Note that for the Stokes problem, with \(\operatorname{div}\boldsymbol{u} = 0\), we can simplify\footnote{Using \(\operatorname{Div}\operatorname{stf} \textnormal{D} \boldsymbol{u} = \operatorname{Div}[\tfrac{1}{2} \text{D} \boldsymbol{u} + \tfrac{1}{2} (\text{D} \boldsymbol{u})^{\top} - \tfrac{\operatorname{div} \boldsymbol{u}}{3}\bbone_3]\), \(\operatorname{Div} [(\text{D} \boldsymbol{u})^{\top}] = \boldsymbol{\nabla} (\operatorname{div} \boldsymbol{u})\), and \(\operatorname{Div} \text{D}\boldsymbol{u} = \Delta \boldsymbol{u}\).} \(-2 \operatorname{Div}\operatorname{stf} \textnormal{D} \boldsymbol{u} = - \Delta \boldsymbol{u}\), resembling the classical notation. On the boundary \(\Gamma\), with \(\tilde{\chi} = 1\) for simplicity, taking only a subset of the full BCs \cref{bcs_3d} as
\begin{equation}\label{bcs_3d_subset_stokes}
  \begin{aligned}
    \left(u_n - u_n^{\mathrm{w}}\right) = \epsilon^{\mathrm{w}} \left((p-p^{\mathrm{w}}) + \sigma_{nn}\right)
    ,
    \quad \textnormal{ and } \quad
    \sigma_{nt_{i}} = u_{t_i}-u_{t_i}^{\mathrm{w}} \quad (i \in \{1,2\})
    ,
  \end{aligned}
\end{equation}
for the Stokes problem yields a slip (or no-slip if \(\epsilon^{\textnormal{w}} = 0\)) condition for \(u_n\) and a (Navier- or friction-type) slip condition for the tangential stress. For the Poisson problem, the subset
\begin{equation}\label{bcs_3d_subset_poisson}
  \begin{aligned}
    s_n &= 2(\theta-\theta^{\mathrm{w}})
    ,
  \end{aligned}
\end{equation}
of \cref{bcs_3d} resembles a Robin-type condition. In short, the R13 equations extend the classical Stokes and Poisson problems to cases where the \(\mathcal{O}(\operatorname{Kn}^2)\)-contributions are non-negligible and further coupling effects in \(\Omega\), as well as on \(\Gamma\), play a role.

\subsection{Additional notation}\label{subsec:notation}
We will denote by $\skalarProd{}{}$, $ \boldsymbol{:}$, and $\pmb{\because}$ the scalar product of vectors, matrices, and $3$-tensors, respectively. We annotate all scalar products and norms if they are not clear from the context (e.g., \(\|a\|\) denotes the operator norm of the bilinear form $a$). To emphasize the slightly different action of the divergence operator on vector, matrix, and $3$-tensor fields, we will use the notations $\div$, $\Div$, and $\DIV$, respectively. However, the derivative (the Jacobian matrix) will always be denoted by $\D\cdot=\cdot\otimes\boldsymbol{\nabla}$. For a square matrix, $\boldsymbol{P}\in\R^{3\times3}$, we denote by $\boldsymbol{P}^\top$ its transpose, by $\tr \boldsymbol{P}\coloneqq \boldsymbol{P}\boldsymbol{:}\bbone_3$ its trace, by $\dev \boldsymbol{P} \coloneqq \boldsymbol{P} - \frac{\tr \boldsymbol{P}}{3}\bbone_3$ its deviatoric or trace-free part, by $\sym \boldsymbol{P}\coloneqq \frac12(\boldsymbol{P}+\boldsymbol{P}^\top)$ its symmetric part, and by $\stf \boldsymbol{P} \coloneqq \dev\sym \boldsymbol{P}$ its symmetric and trace-free part. Thus, in particular, we are interested in orthogonal projections $\mathscr{A}:\R^{3\times 3} \to \R^{3\times 3}$, e.g., $\mathscr{A}\in\{\dev,\sym,\stf\}$, and the corresponding subspaces $\R^{3\times 3}_{\mathscr{A}}\coloneqq\{\boldsymbol{P}\in\R^{3\times3}:\mathscr{A}[\boldsymbol{P}]=\boldsymbol{P}\}$. Similarly, we introduce the deviatoric part and explicitly write out the previously mentioned symmetric part as well as the symmetric, trace-free part of the $3$-tensor $\boldsymbol{P} = (P_{ijk})_{i,j,k = 1}^{3} \in\R^{3\times3\times3}$ in coordinate-wise notation as
\begingroup{}%
\allowdisplaybreaks[1]%
\begin{subequations}\begin{align}
  {(\Dev \boldsymbol{P})}_{ijk}
   &\coloneqq
   P_{ijk} - \tfrac{1}{5} \left( P_{ill} \delta_{jk} + P_{ljl} \delta_{ik} + P_{llk} \delta_{ij}\right),\\
  {(\Sym \boldsymbol{P})}_{ijk}
   &\coloneqq
   P_{(ijk)}
   \coloneqq\tfrac16\left(P_{ijk}+P_{jki}+P_{kij}+P_{jik}+P_{ikj}+P_{kji}\right)
    ,\\
   {(\Stf \boldsymbol{P})}_{ijk}
   &\coloneqq
   P_{\langle ijk \rangle}
   \coloneqq
   P_{(ijk)} - \tfrac{1}{5} \left( P_{(ill)} \delta_{jk} + P_{(ljl)} \delta_{ik} + P_{(llk)} \delta_{ij}\right).
\end{align}\end{subequations}
\endgroup{}%


\section{A mixed formulation for the linear R13 equations}\label{s_mixedform}
For the analysis, recognizing the saddle point structure is crucial. Therefore, let us first introduce the Hilbert spaces as
\begin{align}\label{eq_functionSpaceV}
  \mathrm{V} \coloneqq \hil^1(\Omega;\mathbb{R}^{3 \times 3}_{\operatorname{stf}}) \times \hil^1(\Omega;\mathbb{R}^3) \times \tilde{\operatorname{H}}{}^1(\Omega;\mathbb{R})
  ,\quad
  \mathrm{Q} \coloneqq \lebe^2(\Omega;\mathbb{R}^3) \times \lebe^2(\Omega;\mathbb{R})
  ,
\end{align}
together with their canonical norms \({\|\cdot\|}_\mathrm{V}\) and \({\|\cdot\|}_\mathrm{Q}\), whereas we distinguish between \(\tilde{\operatorname{H}}{}^1(\Omega;\mathbb{R}) \coloneqq \left\{p \in \hil^1(\Omega;\mathbb{R}) : \int_{\Omega} p \,\mathrm{d} \boldsymbol{x} = 0\right\}\) if \(\epsilon^{\textnormal{w}} = 0\) (with \(\epsilon^{\textnormal{w}}\) from \cref{bcs_3d}) and \(\tilde{\operatorname{H}}{}^1(\Omega;\mathbb{R}) \coloneqq \operatorname{H}^1(\Omega;\mathbb{R})\) if \(\epsilon^{\textnormal{w}} > 0\).
\par
We can then introduce for all solution fields \((\boldsymbol{\sigma},\boldsymbol{s},p,\boldsymbol{u},\theta)\) from \cref{s_intro}, corresponding test functions \((\boldsymbol{\psi},\boldsymbol{r},q,\boldsymbol{v},\kappa)\), which we group into two sets of variables
\begin{alignat}{2}
  \boldsymbol{\mathcal{U}} &= (\boldsymbol{\sigma},\boldsymbol{s},p) \in \mathrm{V}
  , \quad &
  \boldsymbol{\mathcal{P}} &= (\boldsymbol{u},\theta) \in \mathrm{Q}
  ,\label{eq_grouping1}
  \\
  \boldsymbol{\mathcal{V}} &= (\boldsymbol{\psi},\boldsymbol{r},q) \in \mathrm{V}
  , \quad &
  \boldsymbol{\mathcal{Q}} &= (\boldsymbol{v},\kappa) \in \mathrm{Q}
  .\label{eq_grouping2}
\end{alignat}
To derive the weak form, we integrate and test \cref{eq_balance1,eq_balance2} with their corresponding test functions, apply integration by parts, insert the boundary conditions \cref{bcs_3d}, and use the closure relations \cref{eq_closures}. As the calculations are very lengthy, we refer to~\cite{theisenFenicsR13TensorialMixed2021} for the detailed derivation in the synthetic two-dimensional case (using a slab geometry) with a straightforward extension to three dimensions. We, thus, obtain the abstract mixed formulation: Given \(\mathcal{F} \in \mathrm{V}', \mathcal{G} \in \mathrm{Q}'\), find \(\boldsymbol{\mathcal{U}} \in \mathrm{V}, \boldsymbol{\mathcal{P}} \in \mathrm{Q}\) such that
\begin{alignat}{2}\label{eq_mixedFormAbstract1a}
  \mathcal{A}(\boldsymbol{\mathcal{U}},\boldsymbol{\mathcal{V}})+ \mathcal{B}(\boldsymbol{\mathcal{V}},\boldsymbol{\mathcal{P}}) &= \mathcal{F}(\boldsymbol{\mathcal{V}}) \quad &\forall\, \boldsymbol{\mathcal{V}}{} &\in \mathrm{V},
  \\
  \label{eq_mixedFormAbstract2a}
  \mathcal{B}(\boldsymbol{\mathcal{U}},\boldsymbol{\mathcal{Q}}) &= \mathcal{G}(\boldsymbol{\mathcal{Q}}) \quad &\forall\, \boldsymbol{\mathcal{Q}} &\in \mathrm{Q}.
\end{alignat}
In \cref{eq_mixedFormAbstract1a,eq_mixedFormAbstract2a}, the grouped bilinear forms are given by
\begin{align}\label{eq_bilinearFormCalA}
  \mathcal{A}(\boldsymbol{\mathcal{U}},\boldsymbol{\mathcal{V}})
  &=
  a(\boldsymbol{s},\boldsymbol{r})
  +
  c(\boldsymbol{s},\boldsymbol{\psi})
  -
  c(\boldsymbol{r},\boldsymbol{\sigma})
  +
  d(\boldsymbol{\sigma},\boldsymbol{\psi}) + f(p,\boldsymbol{\psi}) + f(q,\boldsymbol{\sigma}) + h(p,q)
  ,
  \\
  \label{eq_bilinearFormCalB}
  \mathcal{B}(\boldsymbol{\mathcal{U}},\boldsymbol{\mathcal{Q}})
  &=
  -
  e(\boldsymbol{v},\boldsymbol{\sigma})
  -
  g(p,\boldsymbol{v})
  -
  b(\kappa, \boldsymbol{s})
  ,
\end{align}
which are defined, with \(\tilde{\chi}\) and \(\epsilon^{\textnormal{w}}\) from \cref{bcs_3d}, by the bilinear forms
\begingroup{}%
\allowdisplaybreaks[1]%
\begin{subequations}
\begin{align}
  a(\boldsymbol{s},\boldsymbol{r})
  &= \begin{multlined}[t]
  \frac{24}{25} \operatorname{Kn} \int_\Omega \operatorname{sym} \textnormal{D}\boldsymbol{s} \boldsymbol{:} \operatorname{sym} \textnormal{D}\boldsymbol{r} \,\textnormal{d} \boldsymbol{x}
  +
  \frac{12}{25} \operatorname{Kn} \int_\Omega \operatorname{div}(\boldsymbol{s}) \operatorname{div}(\boldsymbol{r}) \,\textnormal{d} \boldsymbol{x}
  \\
  +
  \frac{4}{15} \frac{1}{\operatorname{Kn}} \int_\Omega \boldsymbol{s} \boldsymbol{\cdot} \boldsymbol{r} \,\textnormal{d} \boldsymbol{x}
  + \frac{1}{2} \frac{1}{\tilde{\chi}} \int_\Gamma s_n r_n \,\textnormal{d} l
  + \frac{12}{25} \tilde{\chi} \sum_{i=1}^2 \int_\Gamma s_{t_i} r_{t_i} \,\textnormal{d} l
  ,
  \end{multlined}
  \\
  c(\boldsymbol{r},\boldsymbol{\sigma})
  &=
  \frac{2}{5} \int_\Omega \boldsymbol{\sigma} \boldsymbol{:} \textnormal{D} \boldsymbol{r} \,\textnormal{d} \boldsymbol{x}
  - \frac{3}{20} \int_\Gamma \sigma_{nn} r_n \,\textnormal{d} l
  - \frac{1}{5} \sum_{i=1}^2 \int_\Gamma \sigma_{nt_{i}} r_{t_i} \,\textnormal{d} l
  ,
  \\
  d(\boldsymbol{\sigma},\boldsymbol{\psi})
  &= \begin{multlined}[t]
  \operatorname{Kn} \int_\Omega \operatorname{Stf} \textnormal{D} \boldsymbol{\sigma} \pmb{\because} \operatorname{Stf} \textnormal{D} \boldsymbol{\psi} \,\textnormal{d} \boldsymbol{x}
  +
  \frac{1}{2} \frac{1}{\operatorname{Kn}} \int_\Omega \boldsymbol{\sigma} \boldsymbol{:} \boldsymbol{\psi} \,\textnormal{d} \boldsymbol{x}
  \\
  +
  \frac{9}{8} \tilde{\chi} \int_\Gamma \sigma_{nn} \psi_{nn} \,\textnormal{d} l
  +
  \tilde{\chi} \int_\Gamma \left( \sigma_{t_{1}t_{1}} + \frac{1}{2} \sigma_{nn} \right) \left( \psi_{t_{1}t_{1}} + \frac{1}{2} \psi_{nn} \right) \,\textnormal{d} l
  \\
  +
  \tilde{\chi} \int_\Gamma \sigma_{t_{1}t_{2}} \psi_{t_{1}t_{2}} \,\textnormal{d} l
  +
  \frac{1}{\tilde{\chi}}\sum_{i=1}^2  \int_\Gamma \sigma_{nt_{i}} \psi_{nt_{i}} \,\textnormal{d} l
  +
  \epsilon^{\textnormal{w}} \tilde{\chi} \int_\Gamma \sigma_{nn} \psi_{nn} \,\textnormal{d} l
  \label{eq:bilinaerform_d}
  ,
  \end{multlined}
  \\
  b(\theta, \boldsymbol{r})
  &=
  \int_\Omega \theta \, \operatorname{div}(\boldsymbol{r}) \,\textnormal{d} \boldsymbol{x}
  ,\quad
  e(\boldsymbol{u},\boldsymbol{\psi})
  =
  \int_\Omega \operatorname{Div}(\boldsymbol{\psi}) \boldsymbol{\cdot} \boldsymbol{u} \,\textnormal{d} \boldsymbol{x}
  ,
  \\
  f(p,\boldsymbol{\psi})
  &=
  \epsilon^{\textnormal{w}} \tilde{\chi} \int_\Gamma p \psi_{nn} \,\textnormal{d} l
  ,\quad
  g(p,\boldsymbol{v})
  =
  \int_\Omega \boldsymbol{v} \boldsymbol{\cdot} \boldsymbol{\nabla} p \,\textnormal{d} \boldsymbol{x}
  \label{eq:bilinaerform_fg}
  ,
  \\
  h(p,q)
  &=
  \epsilon^{\textnormal{w}} \tilde{\chi}  \int_\Gamma p q \,\textnormal{d} l
  \label{eq:bilinaerform_h}
  .
\end{align}
\end{subequations}
\endgroup{}%
The linear functionals in \cref{eq_mixedFormAbstract1a,eq_mixedFormAbstract2a}, on the other hand, read
\begin{align}
  \mathcal{F}(\boldsymbol{\mathcal{V}})
  &=
  l_1(\boldsymbol{r})
  +
  l_3(\boldsymbol{\psi})
  +
  l_5(q)
  ,
  \\
  \mathcal{G}(\boldsymbol{\mathcal{Q}})
  &=
  -
  l_2(\kappa)
  -
  l_4(\boldsymbol{v})
  ,
\end{align}
with the functionals
\begin{subequations}\begin{align}
  l_1(\boldsymbol{r}) &= - \int_\Gamma \theta^{\textnormal{w}} r_n \,\textnormal{d} l
  ,\quad
  l_2(\kappa) = \int_\Omega \left( r - m \right) \kappa \,\textnormal{d} \boldsymbol{x}
  ,
  \\
  l_3(\boldsymbol{\psi}) &= - \int_\Gamma \left( \sum_{i=1}^2 u_{t_i}^{\textnormal{w}} \psi_{nt_{i}} + \left( u_n^{\textnormal{w}} - \epsilon^{\textnormal{w}} \tilde{\chi} p^{\textnormal{w}} \right) \psi_{nn} \right) \,\textnormal{d} l
  ,\quad
  l_4(\boldsymbol{v}) = \int_\Omega \boldsymbol{b} \boldsymbol{\cdot} \boldsymbol{v} \,\textnormal{d} \boldsymbol{x}
  ,
  \\
  l_5(q) &= \int_\Omega m q \,\textnormal{d} \boldsymbol{x} - \int_\Gamma \left( u_n^{\textnormal{w}} - \epsilon^{\textnormal{w}} \tilde{\chi}  p^{\textnormal{w}} \right) q \,\textnormal{d} l
  .
\end{align}\end{subequations}%
\par
\begin{remark}
  To simplify notation, later on, adding \(d(\boldsymbol{\sigma}, \boldsymbol{\psi})\), \(f(p,\boldsymbol{\psi})\), \(f(q,\boldsymbol{\sigma})\), and \(h(p,q)\) from \cref{eq:bilinaerform_d,eq:bilinaerform_fg,eq:bilinaerform_h} allows us to define \(\bar{d} : [\hil^1(\Omega;\mathbb{R}^{3 \times 3}_{\operatorname{stf}}) \times \tilde{\operatorname{H}}{}^1(\Omega;\mathbb{R})] \times [\hil^1(\Omega;\mathbb{R}^{3 \times 3}_{\operatorname{stf}}) \times \tilde{\operatorname{H}}{}^1(\Omega;\mathbb{R})] \to \mathbb{R}\) as
  \begin{equation}
    \begin{multlined}[t]
    \bar{d}((\boldsymbol{\sigma}, p), (\boldsymbol{\psi}, q))
    =
    \operatorname{Kn} \int_\Omega \operatorname{Stf} \textnormal{D} \boldsymbol{\sigma} \pmb{\because} \operatorname{Stf} \textnormal{D} \boldsymbol{\psi} \,\textnormal{d} \boldsymbol{x}
    +
    \frac{1}{2} \frac{1}{\operatorname{Kn}} \int_\Omega \boldsymbol{\sigma} \boldsymbol{:} \boldsymbol{\psi} \,\textnormal{d} \boldsymbol{x}
    \\
    +
    \frac{9}{8} \tilde{\chi} \int_\Gamma \sigma_{nn} \psi_{nn} \,\textnormal{d} l
    +
    \tilde{\chi} \int_\Gamma \left( \sigma_{t_{1}t_{1}} + \frac{1}{2} \sigma_{nn} \right) \left( \psi_{t_{1}t_{1}} + \frac{1}{2} \psi_{nn} \right) \,\textnormal{d} l
    \\
    +
    \tilde{\chi} \int_\Gamma \sigma_{t_{1}t_{2}} \psi_{t_{1}t_{2}} \,\textnormal{d} l
    +
    \frac{1}{\tilde{\chi}}\sum_{i=1}^2  \int_\Gamma \sigma_{nt_{i}} \psi_{nt_{i}} \,\textnormal{d} l
    +
    \epsilon^{\textnormal{w}} \tilde{\chi} \int_\Gamma (p + \sigma_{nn}) (q + \psi_{nn}) \,\textnormal{d} l
    .
  \end{multlined}
    \label{eq:bilinaerform_dbar}
  \end{equation}
  Comparing \(\bar{d}\) to \(d\) from \cref{eq:bilinaerform_d}, only the last term \(\epsilon^{\textnormal{w}} \tilde{\chi} \int_\Gamma (p + \sigma_{nn}) (q + \psi_{nn}) \,\textnormal{d} l\) differs and now contains the \emph{total pressure} \(p+\sigma_{nn}\). This term highlights the positive semidefiniteness of \(\bar{d}\) on the \((\boldsymbol{\sigma},p)\)-variables. Then, \(\mathcal{A}\) from \cref{eq_bilinearFormCalA} is rewritten as
  \begin{equation}
    \mathcal{A}(\boldsymbol{\mathcal{U}},\boldsymbol{\mathcal{V}}) = a(\boldsymbol{s},\boldsymbol{r}) + c(\boldsymbol{s},\boldsymbol{\psi}) - c(\boldsymbol{r},\boldsymbol{\sigma}) + \bar{d}((\boldsymbol{\sigma}, p), (\boldsymbol{\psi}, q))
    .
  \end{equation}
\end{remark}
Using the Cauchy--Bunyakovsky--Schwarz (CBS) inequality, we conclude that both bilinear forms \(\mathcal{A}\) and \(\mathcal{B}\) are continuous (see Appendix~\ref{sec:Appendix}). At the same time, \(\mathcal{A}\) is not entirely symmetric as \(\mathcal{A}(\boldsymbol{\mathcal{U}},\boldsymbol{\mathcal{V}}) = \mathcal{A}(\boldsymbol{\mathcal{V}},\boldsymbol{\mathcal{U}}) + 2 (c(\boldsymbol{r},\boldsymbol{\sigma}) - c(\boldsymbol{s},\boldsymbol{\psi}))\), as only  \(a\) and \(\bar{d}\) are symmetric. We also introduce corresponding linear continuous operators \(A : \mathrm{V} \to \mathrm{V}'\), defined by \({\langle A \boldsymbol{\mathcal{U}} , \boldsymbol{\mathcal{V}} \rangle}_{\mathrm{V}' \times \mathrm{V}} \coloneqq \mathcal{A}(\boldsymbol{\mathcal{U}},\boldsymbol{\mathcal{\mathrm{V}}}) ,\forall\, \boldsymbol{\mathcal{U}}, \boldsymbol{\mathcal{V}} \in \mathrm{V}\), and \(B : \mathrm{V} \to \mathrm{Q}'\), defined by \({\langle B \boldsymbol{\mathcal{V}} , \boldsymbol{\mathcal{Q}} \rangle}_{\mathrm{Q}' \times \mathrm{Q}} \coloneqq \mathcal{B}(\boldsymbol{\mathcal{V}},\boldsymbol{\mathcal{Q}}) ,\forall\, \boldsymbol{\mathcal{V}} \in \mathrm{V}, \forall \boldsymbol{\mathcal{Q}} \in \mathrm{Q}\), which are needed later on.
\begin{figure}[t]%
  \centering%
  \includegraphics{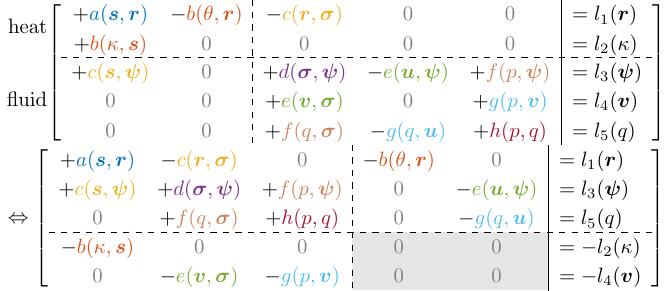}%
  \caption{%
    Visualization of the weak equation structure, in which for the first system (top), the two equations of the heat system only couple through the bilinear form \(c(\boldsymbol{r},\boldsymbol{\sigma})\) to the three fluid equations. A reordering according to trivial diagonal terms yields the saddle point structure in the second system (bottom).%
  }\label{fig:saddlePointStructure}%
\end{figure}%
\par
The choice of \(\mathcal{A}\) and \(\mathcal{B}\) reorders the physical heat (\(\theta,\boldsymbol{s}\)) and fluid (\(p,\boldsymbol{u},\boldsymbol{\sigma}\)) variables into a saddle point system; see \cref{fig:saddlePointStructure} for a sketch of the process. This grouping allows us to describe the system as compactly as possible. Still, this may not be the only possible strategy, and the separate analysis of the heat and fluid equations, coupled only with the bilinear form \(c\), might also be possible (see, e.g., the discussion in~\cite[Sec.~4.3.1]{boffiMixedFiniteElement2013}). Having established the saddle point formulation, we can proceed with the well-posedness analysis. For that, the two main ingredients are:
\begin{enumerate}
  \item the coercivity of \(\mathcal{A}\) on \(\ker B\) (see \cref{sec:CoerKern});
  \item and the inf--sup condition for \(\mathcal{B}\) (see \cref{sec:supCond}).
\end{enumerate}
Both of the above points rely on new ellipticity results for tensor fields, which we will derive in the next section.


\section{Coercivity and Korn-type inequalities}\label{sec:Korn}
Korn inequalities play a fundamental role in the study of variational principles in linear elasticity or fluid mechanics. In their most famous version, they show that the $\lebe^2$-norm of the full gradient of a vector field can be controlled by the $\lebe^2$-norm of only its symmetric part under suitable additional conditions, such as zero boundary values (see, e.g.,~\cite{ciarletKornsInequality2010,kornUberEinigeUngleichungen1909,payneKornsInequality1961}). More precisely, let us denote by $\Omega\subset\R^d$ a bounded Lipschitz domain, then there exists a constant $c=c(\Omega)>0$ such that for all $u\in\hil^1_0(\Omega;\R^d)$, we have
\begin{align}\label{eq:korn1}\tag{Korn-1}
 \norm{\D \boldsymbol u}_{\lebe^2(\Omega)}\le c\,\norm{\sym\D\boldsymbol u}_{\lebe^2(\Omega)}.
\end{align}
The complete classification of differential operators that might appear on the right-hand side of such types of inequalities is already contained in~\cite{aronszajnCoerciveIntegrodifferentialQuadratic1955}. We will refer to such coercivity estimates when zero boundary conditions are present as \emph{Korn inequalities of the first type} and include Aronszajn's characterization for completeness (only focusing on first-order differential operators):
\begin{lemma}[Aronszajn \cite{aronszajnCoerciveIntegrodifferentialQuadratic1955}, Korn inequalities of the first type]\label{lem:Aronszajn}%
  Let $d\ge2$, $\Omega\subset\R^d$ be a bounded Lipschitz domain, $q\in(1,\infty)$, $V$ and $\tilde{V}$ two finite-dimensional real inner product spaces, and $\mathbb{A}\coloneqq\sum_{j=1}^d\mathbb{A}_j\partial_j$ a linear homogeneous differential operator of first order with constant coefficients $\mathbb{A}_j\in L(V,\tilde{V})$. Then the following are equivalent:
\begin{enumerate}
  \item There exists a constant $c=c(\mathbb{A},q,\Omega)>0$ such that the inequality
\begin{align}\tag{K-1}
\norm{\D u}_{\lebe^{q}(\Omega)}\le c\, \norm{\mathbb{A} u}_{\lebe^q(\Omega)}
,
\end{align}
holds for all $u\in\sobo^{1,q}_0(\Omega;V)$.
\item $\mathbb{A}$ is an $\mathbb{R}$-elliptic differential operator, meaning that \phantom{fix latexml on arxiv}  
\begin{align}\tag{\mbox{$\R$-ellipt.}}\label{eq:Rellipt}
 \forall\, \boldsymbol{\xi}\in\R^d\setminus\{\boldsymbol{0}\}: \Bigg(\sum_{j=1}^d\xi_j\mathbb{A}_jv\overset{!}{=} 0 \, \Rightarrow \, v=0 \Bigg).
\end{align}
\end{enumerate}
\end{lemma}

\begin{example}%
  Set $V=\R^d$, $\tilde{V}=\R^{d\times d}$, and consider the symmetrized gradient for vector fields, given by
  \begin{align}
    \mathbb{A}\boldsymbol{u}\coloneqq \sym \D \boldsymbol{u}.
  \end{align}
It is well known that the symmetrized gradient is $\R$-elliptic. Thus, by \cref{lem:Aronszajn}, we recover \cref{eq:korn1}.
\end{example}
\begin{example}%
 In many applications, the differential operator of interest has more structure and is given by applying an orthogonal projection on the gradient of a vector field as
 \begin{align}
  \mathbb{A}\boldsymbol u=\mathscr{A}[\D \boldsymbol u] \quad \text{with}\quad  \mathscr{A}:\R^{d\times d}\to\R^{d\times d} \quad \text{such that} \quad \mathscr{A}^2=\mathscr{A}.
 \end{align}
Classical examples are the symmetric part $\mathscr{A}=\sym$, the deviatoric (trace-free) part $\mathscr{A}=\dev$, or the symmetric trace-free part $\mathscr{A}=\stf=\sym\dev$, which all give $\R$-elliptic operators.
\end{example}

Later in \cref{sec:rightInverseOfDivergence}, we also need the following connection to the classical elliptic regularity theory:

\begin{remark}[Legendre--Hadamard ellipticity]\label{rem:LH}
 The projection $\mathscr{A}:\R^{d\times d}\to\R^{d\times d}$ induces an elliptic differential operator $\mathbb{A}\coloneqq\mathscr{A}[\boldsymbol{\cdot}\otimes\boldsymbol{\nabla}]$ if and only if for all $\boldsymbol{\xi}\in\R^d\setminus\{\boldsymbol{0}\}$, the corresponding symbol map $\mathbb{A}[\boldsymbol{\xi}] \coloneqq \mathscr{A}[\boldsymbol{\cdot}\otimes\boldsymbol{\xi}]:\R^d\to\R^{d\times d}$ is injective. In that case, since $\mathscr{A}$ is continuous and $\partial{B_1(\boldsymbol{0})}\times\partial{B_1(\boldsymbol{0})}$ is compact, we find a $\lambda>0$ such that for all $\boldsymbol{\eta},\boldsymbol{\xi}\in\R^d\setminus\{\boldsymbol{0}\}$, we have
 \begin{align}
  \lambda\le \left\lvert\mathscr{A}\left[\frac{\boldsymbol{\eta}}{\abs{\boldsymbol{\eta}}}\otimes\frac{\boldsymbol{\xi}}{\abs{\boldsymbol{\xi}}}\right]\right\rvert \quad \text{by just taking} \quad \lambda\coloneqq\min_{\substack{\abs{\boldsymbol{z}}=1\\\abs{\boldsymbol{y}}=1}}\abs{\mathscr{A}[\boldsymbol{z}\otimes \boldsymbol{y}]}
  ,
 \end{align}
while $\lambda>0$ holds by the injectivity of $\mathscr{A}[\boldsymbol{\cdot}\otimes\boldsymbol{\xi}]$. Thus, by the linearity of $\mathscr{A}$, we conclude that for all $\boldsymbol{\eta},\boldsymbol{\xi}\in\R^d$:
\begin{align}\label{eq:inject}
 \lambda\abs{\boldsymbol{\eta}}\abs{\boldsymbol{\xi}}\le\abs{\mathscr{A}[\boldsymbol{\eta}\otimes\boldsymbol{\xi}]}.
\end{align}
Since $\mathscr{A}$ is an orthogonal projection, we have
\begin{align}
 \mathscr{A}[\boldsymbol{\eta}\otimes\boldsymbol{\xi}]\boldsymbol{:}\boldsymbol{\eta}\otimes\boldsymbol{\xi}= \mathscr{A}[\boldsymbol{\eta}\otimes\boldsymbol{\xi}]\boldsymbol{:} \mathscr{A}[\boldsymbol{\eta}\otimes\boldsymbol{\xi}]=\abs{ \mathscr{A}[\boldsymbol{\eta}\otimes\boldsymbol{\xi}]}^2\overset{\mathclap{\eqref{eq:inject}}}{\ge}\lambda^2\abs{\boldsymbol{\eta}}^2\abs{\boldsymbol{\xi}}^2
 ,
\end{align}
which is the Legendre--Hadamard ellipticity condition.
\end{remark}
The well-posedness of our model strongly relies on a \emph{Korn inequality of the second type}, i.e., when the boundary conditions are not \emph{essentially} imposed in the function space but are instead \emph{naturally} imposed in the weak form. We recall the following classification of differential operators:
\begin{lemma}[Nečas {\cite[Thm.~5]{necasNormesEquivalentesDans1966}}, Korn inequalities of the second type]\label{lem:Necas}%
  Let $d\ge2$, $\Omega\subset\R^d$ be a bounded Lipschitz domain, $q\in(1,\infty)$, $V$, and $\tilde{V}$ two finite-dimensional real inner product spaces, and $\mathbb{A}\coloneqq\sum_{j=1}^d\mathbb{A}_j\partial_j$ a linear homogeneous differential operator of first order with constant coefficients $\mathbb{A}_j\in L(V,\tilde{V})$. Then the following are equivalent:
\begin{enumerate}
  \item There exists a constant $c=c(\mathbb{A},q,\Omega)>0$ such that the inequality
\begin{align}\tag{K-2}
\norm{u}_{\sobo^{1,q}(\Omega)}\le c\,(\norm{u}_{\lebe^q(\Omega)}+\norm{\mathbb{A} u}_{\lebe^q(\Omega)})
,
\end{align}
holds for all $u\in\sobo^{1,q}(\Omega;V)$.
\item $\mathbb{A}$ is a $\C$-elliptic differential operator, meaning that \phantom{fix latexml on arxiv abcdef}  
\begin{align}\tag{\mbox{$\C$-ellipt.}}\label{eq:Cellipt}
\forall\, \boldsymbol{\xi}\in\C^d\setminus\{\boldsymbol{0}\}: \Bigg(\sum_{j=1}^d\xi_j\mathbb{A}_jv+\mathrm{i}\,\xi_j\mathbb{A}_jw\overset{!}{=} 0 \, \Rightarrow \, v,w=0 \Bigg).
\end{align}
\end{enumerate}
\end{lemma}
A similar result is also valid for higher-order differential operators and norms of negative Sobolev spaces; see~\cite[Cor.\ 7.3]{smithFormulasRepresentFunctions1970}. It is worth noting that such results for $\C$-elliptic differential operators also hold on John domains (cf.~\cite{GmeinederDieningPotAna24}), which need not allow for a boundary trace operator.
\begin{example}%
 Set $V=\R^d$, $\tilde{V}=\R^{d\times d}$, and consider the symmetrized gradient for vector fields as
 \begin{align}
  \mathbb{A}\boldsymbol{u}\coloneqq \sym \D \boldsymbol{u}.
 \end{align}
It is well known that the symmetrized gradient is $\C$-elliptic. Thus, by \cref{lem:Necas}, we find a constant $c>0$, such that for all $u\in\hil^1(\Omega)$, we have
\begin{align}\label{eq:classKorn}\tag{Korn-2}
 \norm{\boldsymbol u}_{\hil^1(\Omega)}\le c\,(\norm{\boldsymbol u}_{\lebe^2(\Omega)}+\norm{\sym\D \boldsymbol u}_{\lebe^2(\Omega)})
 ,
\end{align}
which is the classical Korn inequality of the second type.
\end{example}
\begin{remark}
Notably, several approaches exist to prove coercivity estimates of the type in \cref{lem:Necas,lem:Aronszajn}. Among them, one particularly elegant and concise approach consists of the following three major ingredients:
\begin{enumerate}
  \item An algebraic identity, which in the context of \eqref{eq:classKorn} reads
  \begin{equation}
    \mathrm{D}^2 \boldsymbol u = L(\mathrm{D}\operatorname{sym}\mathrm{D}\boldsymbol u)
    ,
  \end{equation}
  meaning that all the entries of the second derivative of a vector field $\boldsymbol u$ can be expressed as linear combinations of the entries of the derivative of its symmetrized gradient, which is a key feature of $\mathbb{C}$-elliptic operators~\cite{smithFormulasRepresentFunctions1970}.
  \item A Lions lemma~\cite[Thm.\ III.3.2]{duvautInequationsMecaniquePhysique1972}, respectively, Nečas estimate~\cite{necasNormesEquivalentesDans1966}, which is a rather deep result and expresses an equivalence of norms.
  \item A compact embedding, which for \eqref{eq:classKorn} reads: $\operatorname{L}^2(\Omega)\subset\subset\operatorname{H}^{-1}(\Omega)$.
\end{enumerate}
Such an argumentation was first presented in~\cite[Sec.\ III.3.3]{duvautInequationsMecaniquePhysique1972} to prove \eqref{eq:classKorn}. It was later advocated in many different contexts, e.g., in the incompatible case, cf.~\cite{LewNeffKornRev,lewintanKornInequalitiesIncompatible2021} and the references therein.
A stronger version of this Lions lemma was given in~\cite{AmroucheGirault} with a Nečas inequality and the Peetre--Tartar lemma.  Indeed, the Lions lemma is a deep result and is connected to other fundamental results, see~\cite{AmroucheLionslemma}. Furthermore, the Peetre--Tartar lemma can be used
in order to establish the necessity of $\mathbb{C}$-ellipticity in such type of estimates since $\mathbb{C}$-elliptic differential operators can also be characterized by a finite-dimensional kernel, see e.g.~\cite{GmeinederDieningPotAna24}.
\end{remark}

\subsection{Tensor-valued Korn inequality}
For the R13 model in the previous notation, we have the dimension $d=3$, the integrability $q=2$, $V=\R^{3\times 3}_{\stf}\coloneqq\{\boldsymbol{P}\in\R^{3\times3}:\boldsymbol{P}=\operatorname{dev}\operatorname{sym}\boldsymbol{P}=\stf \boldsymbol{P}\}$, $\tilde{V}=\R^{3\times3\times3}$, and consider the differential operator $\mathbb{A}\coloneqq\Stf\circ\, \D=\Stf(\, \cdot\otimes\boldsymbol{\nabla})$, where we recall the symmetric and trace-free part of a $3$-tensor $\boldsymbol{P}\in\R^{3\times3\times3}$, defined by
 \begin{equation}\label{eq:stfFormula}
  {(\Stf \boldsymbol{P})}_{ijk}
  =
  P_{\langle ijk \rangle}
  =
  P_{(ijk)} - \tfrac{1}{5} \left( P_{(ill)} \delta_{jk} + P_{(ljl)} \delta_{ik} + P_{(llk)} \delta_{ij}\right).
\end{equation}
Thus, to check the $\C$-ellipticity of \(\mathbb{A}\), we start with a symmetric and trace-free tensor $\boldsymbol{T}$ and consider the vanishing symbol of its symmetric, trace-free gradient \cref{eq:stfFormula}, which simplifies with \(\tr \boldsymbol{T} = 0\) to
 \begin{align} \label{eq:symbolstfgradstf}
  \tfrac13\left(T_{ij}\xi_k+T_{ik}\xi_j+T_{jk}\xi_i\right)-\tfrac{2}{15}\Big(\sum_{l}T_{lk}\xi_l\delta_{ij}+\sum_{l}T_{lj}\xi_l\delta_{ik}+\sum_{l}T_{li}\xi_{l}\delta_{jk}\Big)=0,
 \end{align}
with arbitrary $\boldsymbol{\xi}\in\C^3\setminus\{\boldsymbol{0}\}$. Our goal is to show \cref{eq:Cellipt}, i.e., to conclude that this equation possesses only the trivial complex-valued solution $\boldsymbol{T}=0$ for all $\boldsymbol{\xi}\in\C^3\setminus\{\boldsymbol{0}\}$. A multiplication of \cref{eq:symbolstfgradstf} with $\xi_k$ and summation over the index $k$ yields
 \begin{multline}
  \tfrac13\Big(T_{ij}\sum_{k}\xi_k\xi_k+\sum_{k}T_{ik}\xi_k\xi_j+\xi_i\sum_{k}T_{jk}\xi_k\Big)
  \\
  -\tfrac{2}{15}\Big(\delta_{ij}\sum_{l,k}T_{lk}\xi_l\xi_k+\xi_i\sum_{l}T_{lj}\xi_l+\sum_{l}T_{li}\xi_{l}\xi_j\Big)=0,
 \end{multline}
which in coordinate-free notations reads as
\begin{align}\label{eq:firststep}
 \tfrac13\left(\boldsymbol{T}(\skalarProd{\boldsymbol{\xi}}{\boldsymbol{\xi}})+(\boldsymbol{T}\boldsymbol{\xi})\otimes\boldsymbol{\xi} +\boldsymbol{\xi}\otimes(\boldsymbol{T}\boldsymbol{\xi})\right)-\tfrac{2}{15}\left(\bbone_3(\skalarProd{(\boldsymbol{T}\boldsymbol{\xi})}{\boldsymbol{\xi}})+\boldsymbol{\xi}\otimes(\boldsymbol{T}\boldsymbol{\xi})+(\boldsymbol{T}\boldsymbol{\xi})\otimes\boldsymbol{\xi}\right)=0.
\end{align}
Note that we use the notation \(\skalarProd{\boldsymbol{\xi}}{\boldsymbol{\zeta}}\coloneqq\sum_{j=1}^3\xi_j\zeta_j\) for all \(\boldsymbol{\xi},\boldsymbol{\zeta}\in\C^3\), which is different from the Hermitian scalar product $\sum_{j=1}^3\xi_j\overline{\zeta_j}$, where $\overline{\,\cdot\,\vphantom{z}}$ denotes the complex conjugate. In particular, this $\skalarProd{}{}$ is \emph{not} non-negative on $\C^3$. Thus, \cref{eq:firststep} simplifies to
\begin{align}
 \tfrac13\boldsymbol{T}(\skalarProd{\boldsymbol{\xi}}{\boldsymbol{\xi}})+\tfrac15\left((\boldsymbol{T}\boldsymbol{\xi})\otimes\boldsymbol{\xi}+\boldsymbol{\xi}\otimes(\boldsymbol{T}\boldsymbol{\xi})\right)-\tfrac{2}{15}\bbone_3(\skalarProd{(\boldsymbol{T}\boldsymbol{\xi})}{\boldsymbol{\xi}})=0. \label{eq:coreCheck}
\end{align}
By multiplying by $\boldsymbol{\xi}^\top$ from the left and $\boldsymbol{\xi}$ from the right, we obtain that
\begin{align}
 (\skalarProd{\boldsymbol{\xi}}{\boldsymbol{\xi}})(\skalarProd{\boldsymbol{\xi}}{(\boldsymbol{T}\boldsymbol{\xi})})\left(\tfrac13+\tfrac25-\tfrac{2}{15}\right)=0.\label{eq:coreCheck1}
\end{align}
We distinguish two cases:
\par
 \textbf{1.\ case} ($\boldsymbol{\xi}\in\C^3\setminus\{\boldsymbol{0}\}$ satisfies $\skalarProd{\boldsymbol{\xi}}{\boldsymbol{\xi}}\neq0$): Then, by \cref{eq:coreCheck1}, we have \(\skalarProd{\boldsymbol{\xi}}{(\boldsymbol{T}\boldsymbol{\xi})}=0\), and multiplying \cref{eq:coreCheck} with $\boldsymbol{\xi}$ from the right yields
\begin{align}\label{eq:case1}
 (\skalarProd{\boldsymbol{\xi}}{\boldsymbol{\xi}})\left(\tfrac13+\tfrac15\right)\boldsymbol{T}\boldsymbol{\xi}=0 \quad \overset{\mathclap{\skalarProd{\boldsymbol{\xi}}{\boldsymbol{\xi}}\neq0}}{\Rightarrow} \quad \boldsymbol{T}\boldsymbol{\xi}=0
 ,
\end{align}
such that \cref{eq:coreCheck} simplifies to \(\boldsymbol{T}(\skalarProd{\boldsymbol{\xi}}{\boldsymbol{\xi}})=0\), and since $\skalarProd{\boldsymbol{\xi}}{\boldsymbol{\xi}}\neq0$, we conclude that \(\boldsymbol{T}=\boldsymbol{0}\).
\par
\textbf{2.\ case} ($\boldsymbol{\xi}\in\C^3\setminus\{\boldsymbol{0}\}$ satisfies $\skalarProd{\boldsymbol{\xi}}{\boldsymbol{\xi}}=0$ but $\abs{\boldsymbol{\xi}}^2\neq0$): Multiplying \cref{eq:coreCheck} with $\boldsymbol{\xi}$ from the right, we obtain
\begin{align}\label{eq:case2}
  \boldsymbol{\xi}(\skalarProd{(\boldsymbol{T}\boldsymbol{\xi})}{\boldsymbol{\xi}})\left(\tfrac15-\tfrac{2}{15}\right)=0
  ,
\end{align}
so that a multiplication with $\overline{\boldsymbol{\xi}}^\top$ from the left yields
\begin{align}
 \abs{\boldsymbol{\xi}}^2\skalarProd{(\boldsymbol{T}\boldsymbol{\xi})}{\boldsymbol{\xi}}=0 \quad \overset{\mathclap{\boldsymbol{\xi}\neq\boldsymbol{0}}}{\Rightarrow}\quad \skalarProd{(\boldsymbol{T}\boldsymbol{\xi})}{\boldsymbol{\xi}}=0.
\end{align}
Hence, \cref{eq:coreCheck} simplifies to
\begin{align} \label{eq:coreCheck2}
 (\boldsymbol{T}\boldsymbol{\xi})\otimes \boldsymbol{\xi} +\boldsymbol{\xi}\otimes (\boldsymbol{T}\boldsymbol{\xi})=0.
\end{align}
Multiplying now with $\overline{\boldsymbol{\xi}}$ results in
\begin{align}
 \abs{\boldsymbol{\xi}}^2\boldsymbol{T}\boldsymbol{\xi}+\boldsymbol{\xi}(\skalarProd{(\boldsymbol{T}\boldsymbol{\xi})}{\overline{\boldsymbol{\xi}}})=0.
\end{align}
Thus, since $\boldsymbol{\xi}\neq \boldsymbol{0}$, we find an $\alpha\in\C$, such that
\begin{align}\label{eq:parallelity}
 \boldsymbol{T}\boldsymbol{\xi}=\alpha\boldsymbol{\xi}.
\end{align}
Reinserting \cref{eq:parallelity} in \cref{eq:coreCheck2} results in \(2\alpha \boldsymbol{\xi}\otimes\boldsymbol{\xi} = 0\), and multiplication with $\overline{\boldsymbol{\xi}}^\top$ from the left and $\overline{\boldsymbol{\xi}}$ from the right yields
\begin{align}
  \alpha\abs{\boldsymbol{\xi}}^4=0\quad \overset{\mathclap{\boldsymbol{\xi}\neq\boldsymbol{0}}}{\Rightarrow}\quad  \alpha=0 \quad \overset{\mathclap{\eqref{eq:parallelity}}}{\Rightarrow}\quad \boldsymbol{T}\boldsymbol{\xi}=0.
\end{align}
Returning now to \cref{eq:symbolstfgradstf}, the three sums vanish since $\boldsymbol{T}\boldsymbol{\xi}=0$, and after a multiplication by $\overline{\xi}_k$ and summation over the index $k$, we obtain
\begin{align}\label{eq:3.21}
  \boldsymbol{T}|\boldsymbol\xi|^2+(\boldsymbol{T}\overline{\boldsymbol\xi})\otimes\boldsymbol{\xi}+\boldsymbol\xi\otimes(\boldsymbol{T}\overline{\boldsymbol\xi})=0.
\end{align}
Hence, a multiplication with $\overline{\boldsymbol\xi}$ from the right gives
\begin{align}\label{eq:3.22}
  2\,\boldsymbol{T}\overline{\boldsymbol\xi}\,|\boldsymbol\xi|^2+\boldsymbol\xi\,((\boldsymbol{T}\overline{\boldsymbol\xi})\boldsymbol\cdot\overline{\boldsymbol\xi})=0.
\end{align}
Multiplying by $\overline{\boldsymbol\xi}^\top$ from the left, we deduce
\begin{align}
  3\,((\boldsymbol{T}\overline{\boldsymbol\xi})\boldsymbol\cdot\overline{\boldsymbol\xi})\,|\boldsymbol\xi|^2=0 \quad \overset{\boldsymbol{\xi}\neq\boldsymbol{0}}{\Rightarrow}\quad
  (\boldsymbol{T}\overline{\boldsymbol\xi})\boldsymbol\cdot\overline{\boldsymbol\xi}=0.
\end{align}
Thus, by \cref{eq:3.22}, it holds $\boldsymbol{T}\overline{\boldsymbol\xi}=0$, and, returning to \cref{eq:3.21}, we conclude that $\boldsymbol{T}=0$. Summarizing, we have shown that for a symmetric trace-free tensor $\boldsymbol{T}$,
\begin{align}
  \tfrac13\left(T_{ij}\xi_k+T_{ik}\xi_j+T_{jk}\xi_i\right)-\tfrac{2}{15}\Big(\sum_{l}T_{lk}\xi_l\delta_{ij}+\sum_{l}T_{lj}\xi_l\delta_{ik}+\sum_{l}T_{li}\xi_{l}\delta_{jk}\Big)\overset{!}{=}0
  ,
\end{align}
for $\boldsymbol{\xi}\in\C^3\setminus\{\boldsymbol{0}\}$ implies that $\boldsymbol{T}=\boldsymbol{0}$, meaning that the symmetric trace-free gradient is $\C$-elliptic on symmetric trace-free tensor fields.

\begin{remark}
Note that skew-symmetric matrices are in the kernel of the symbol of the symmetric trace-free gradient (since, by definition \cref{eq:stfFormula}, we only consider the symmetric and trace-free parts), such that it is not even $\R$-elliptic on $\R^{3\times 3}$. Also, for just symmetric but not trace-free matrices, we do not have $\R$-ellipticity. Consider, for example, \(a \bbone_3 \in \mathbb{R}_{\operatorname{sym}}^{3 \times 3}\) (\(a \in \mathbb{R}\)) such that for \(\boldsymbol{P} = a \bbone_3 \otimes\boldsymbol{\xi}\) (\(\boldsymbol{\xi} \in \mathbb{R}^3\)), we have that \(P_{(ill)} = \tfrac{1}{3}(a \tr \bbone_3 + 2a \bbone_3) \xi_i = \tfrac{5}{3} a \xi_i\). Then, \({(\Stf (a \bbone_3 \otimes\boldsymbol{\xi}))}_{ijk}\) in \cref{eq:stfFormula} vanishes as
\begin{align}
  \tfrac{1}{3} a (\xi_k \delta_{ij} + \xi_j \delta_{ik} + \xi_i \delta_{jk})
  - \tfrac{1}{5} \left( \tfrac{5}{3} a \xi_i \delta_{jk} + \tfrac{5}{3} a \xi_j \delta_{ik} + \tfrac{5}{3} a \xi_k \delta_{ij} \right)
  =
  0
  \quad
  \forall\, \boldsymbol{\xi} \in \mathbb{R}^3
  .
\end{align}
\end{remark}

\begin{lemma}\label{lem:CelliptSTF}%
 The symmetric trace-free gradient on symmetric trace-free $2$-tensor fields in $\R^{d\times d}$ is $\C$-elliptic if and only if $d\geq3$.
\end{lemma}

\begin{proof}
  Besides the case of \(d=3\), in which the previous calculations prove the result, we distinguish the following cases:
  \par
  \textbf{Case $d>3$}: The prefactors have to be changed to
  \begin{multline}
  \tfrac15 \mapsto \tfrac{1}{d+2} \textnormal{ in \cref{eq:stfFormula}}
  ,\quad
  \tfrac{2}{15} \mapsto \tfrac{2}{3(d+2)} \textnormal{ in \cref{eq:symbolstfgradstf}}
  ,\quad
  \tfrac13+\tfrac25-\tfrac{2}{15} \mapsto \tfrac{d}{d+2} \textnormal{ in \cref{eq:coreCheck1}}
  ,
  \\
  \tfrac13+\tfrac15 \mapsto \tfrac{2(d+1)}{3(d+2)} \textnormal{ in \cref{eq:case1}}
  ,\quad
  \tfrac15-\tfrac{2}{15} \mapsto \tfrac{d-2}{3(d+2)} \textnormal{ in \cref{eq:case2}}
  .
  \end{multline}
  None of these prefactors vanishes for $d\ge3$, meaning that we can basically follow the arguments from the case of \(d=3\) to conclude the $\C$-ellipticity of the symmetric trace-free gradient also on symmetric trace-free tensors with $d\ge3$.
  \par
  \textbf{Case $d=2$}: Here, the prefactor in the replacement of \cref{eq:case2} would vanish. In fact, consider, for example, the choice of
  \begin{align}
   \boldsymbol{T}=\begin{pmatrix} \mathrm{i} & 1 \\ 1 & \mathrm{-i} \end{pmatrix} \quad \text{and}\quad \boldsymbol{\xi}=\begin{pmatrix} 1 \\ \mathrm{i} \end{pmatrix},
  \end{align}
 for which we have, for all $i,j,k\in\{1,2\}$, that
  \begin{align}
  \tfrac13\left(T_{ij}\xi_k+T_{ik}\xi_j+T_{jk}\xi_i\right)-\tfrac{1}{6}\Big(\sum_{l}T_{lk}\xi_l\delta_{ij}+\sum_{l}T_{lj}\xi_l\delta_{ik}+\sum_{l}T_{li}\xi_{l}\delta_{jk}\Big)=0
  .
 \end{align}
 Thus, the symmetric trace-free gradient on symmetric trace-free tensors is not $\C$-elliptic in $2$ dimensions.
\end{proof}

\begin{remark}
 Note that the symmetric trace-free gradient on symmetric trace-free tensors is $\R$-elliptic in all dimensions $d\geq2$. For the proof, we follow the arguments above and stop after the first case. The second case is not relevant since \(\boldsymbol{\xi} \in \mathbb{R}^d\) when checking $\R$-ellipticity.
\end{remark}

\begin{remark}
  With the lack of $\C$-ellipticity for \(d=2\) in \cref{lem:CelliptSTF}, we encounter a phenomenon similar
  to the known missing $\C$-ellipticity of the symmetric trace-free gradient for vector fields in two dimensions (cf.~\cite{Reshetnyak,dainGeneralizedKornInequality2006a,GRVS}), though at one tensor rank higher.
\end{remark}

\begin{lemma}[Tensor-valued Korn inequality]\label{lem:tensorKorn}%
 Let $d\ge3$, $\Omega\subset\R^d$ be a bounded Lipschitz domain and $\boldsymbol\sigma\in\hil^1(\Omega;\R^{d\times d}_{\stf})$. There exists a constant $c=c(\Omega)>0$ such that
 \begin{align}
  \norm{\boldsymbol\sigma}_{\hil^1(\Omega)}\le c\left(\norm{\Stf\D\boldsymbol \sigma}_{\lebe^2(\Omega)}+\norm{\boldsymbol\sigma}_{\lebe^2(\Omega)}\right).
 \end{align}
\end{lemma}
\begin{proof}
In the preceding calculations, we showed that the differential operator $\mathbb{A}\coloneqq\Stf\circ\,\D$ is $\C$-elliptic for symmetric, trace-free tensor fields. Thus, the condition \cref{eq:Cellipt} is satisfied, and we conclude by \cref{lem:Necas}.
\end{proof}
With the above estimates, we already obtain the coercivity of the bilinear forms \(a\) and \(\bar{d}\) from \cref{s_mixedform} since
\begin{equation}
\begin{aligned}
 a(\boldsymbol s,\boldsymbol s)
 &=
 \tfrac{24}{25}\operatorname{Kn} \norm{\sym\D \boldsymbol s}_{\lebe^2(\Omega)}^2+\tfrac{12}{25}\operatorname{Kn}\norm{\div \boldsymbol s}_{\lebe^2(\Omega)}^2+\tfrac{4}{15}\tfrac{1}{\operatorname{Kn}}\norm{\boldsymbol s}_{\lebe^2(\Omega)}^2\\
 &\quad
 +\tfrac12\tfrac{1}{\tilde{\chi}}\norm{s_n}_{\lebe^2(\Gamma)}^2
 +\tfrac{12}{25}\tilde{\chi} \textstyle{\sum_{i=1}^2} \norm{s_{t_i}}_{\lebe^2(\Gamma)}^2\\
 &\ge
 \min\left\{\tfrac{24}{25}\operatorname{Kn},\tfrac{4}{15}\tfrac{1}{\operatorname{Kn}}\right\}\left(\norm{\sym\D \boldsymbol s}_{\lebe^2(\Omega)}^2+\norm{\boldsymbol s}_{\lebe^2(\Omega)}^2\right) \\
 &\overset{\mathclap{\eqref{eq:classKorn}}}{\ge}
 c_0 \norm{\boldsymbol s}_{\hil^1(\Omega)}^2, \label{eq:forA}
\end{aligned}
\end{equation}
as well as
\begin{equation}
  \begin{aligned}
 \bar{d}((\boldsymbol\sigma,p),(\boldsymbol\sigma,p))
 &=
 \operatorname{Kn}\norm{\Stf\D\boldsymbol\sigma}_{\lebe^2(\Omega)}^2+\tfrac12\tfrac{1}{\operatorname{Kn}}\norm{\boldsymbol\sigma}_{\lebe^2(\Omega)}^2
 +\tfrac98\tilde{\chi}\norm{\sigma_{nn}}_{\lebe^2(\Gamma)}^2
 \\
 &\quad
 +\tilde{\chi}\norm{\sigma_{t_{1}t_{1}}+\tfrac12\sigma_{nn}}_{\lebe^2(\Gamma)}^2
 +\tilde{\chi}\norm{\sigma_{t_{1}t_{2}}}_{\lebe^2(\Gamma)}^2
 \\
 &\quad
 +\tfrac{1}{\tilde{\chi}} \textstyle{\sum_{i=1}^2} \norm{\sigma_{nt_{i}}}_{\lebe^2(\Gamma)}^2
 +\epsilon^w\tilde{\chi}\norm{p+\sigma_{nn}}_{\lebe^2(\Gamma)}^2
 \\
 &\ge \min\left\{\operatorname{Kn},\tfrac12\tfrac{1}{\operatorname{Kn}}\right\}\left(\norm{\Stf\D\boldsymbol\sigma}_{\lebe^2(\Omega)}^2+\norm{\boldsymbol\sigma}_{\lebe^2(\Omega)}^2\right)
 \\
 &\overset{\mathclap{\text{\cref{lem:tensorKorn}}}}{\ge}
 c_1 \norm{\boldsymbol\sigma}_{\hil^1(\Omega)}^2. \label{eq:forD}
\end{aligned}
\end{equation}

\subsection{Right inverse of the matrix-valued divergence operator}\label{sec:rightInverseOfDivergence}
The ellipticity property of the previous section also yields the second main theoretical tool regarding the inversion of the matrix-valued divergence. But first, we recall the classical result about vector fields (see, e.g.,~\cite[Lem.\ 11.2.3]{brennerMathematicalTheoryFinite2008} or~\cite[Satz 6.3]{Braess}).
\begin{lemma}\label{lem:Invdiv0}%
  Let $d\ge2$, and let $\Omega\subset\R^d$ be a bounded Lipschitz domain. Then there exists a constant $c=c(\Omega)>0$, such that for all $\kappa\in\lebe^2(\Omega;\R)$, we find a $\boldsymbol{t}\in\hil^1(\Omega;\R^{d})$ satisfying
  \begin{align}
   -\div \boldsymbol t = \kappa \quad \text{and}\quad \norm{\boldsymbol t}_{\hil^1(\Omega)}\le c\,\norm{\kappa}_{\lebe^2(\Omega)}.
  \end{align}
 \end{lemma}
Now, we extend this result to higher-order tensors. More precisely, we construct a right inverse of the matrix-valued divergence operator that depends continuously on the data. Since the divergence acts row-wise in
\begin{equation}
  \Div \boldsymbol\tau = \boldsymbol u,,
\end{equation}
one could solve the problem row by row using \cref{lem:Invdiv0}. However, the resulting matrix field would generally not satisfy any additional structure. The key point of the next lemma is that we construct a right inverse of the matrix divergence that still solves the equation while producing a matrix field with the desired structural constraints (e.g., symmetry or symmetry and vanishing trace).
\begin{lemma}\label{lem:InvDiv1}%
  Let $d\ge2$, $\Omega\subset\R^d$ be a bounded Lipschitz domain, and $\mathscr{A}:\R^{d\times d}\to\R^{d\times d}$ be an orthogonal projection, such that $\mathbb{A}\coloneqq\mathscr{A}[\boldsymbol{\cdot}\otimes\boldsymbol{\nabla}]$ is an elliptic differential operator from $\R^d$ to $\R^{d\times d}$. Let us further denote by $\R^{d\times d}_{\mathscr{A}}\coloneqq\{\boldsymbol{\tau}\in\R^{d\times d}:\mathscr{A}[\boldsymbol{\tau}]=\boldsymbol{\tau}\}$. Then there exists a constant $c=c(\mathscr{A},\Omega)>0$, such that for all $\boldsymbol u\in\lebe^2(\Omega;\R^d)$, we find a $\boldsymbol{\tau}\in\hil^1(\Omega;\R^{d\times d}_{\mathscr{A}})$ satisfying
  \begin{align}
   -\Div \boldsymbol\tau = \boldsymbol u \quad \text{and}\quad \norm{\boldsymbol\tau}_{\hil^1(\Omega)}\le c\,\norm{\boldsymbol u}_{\lebe^2(\Omega)}.
  \end{align}
 \end{lemma}
 \begin{proof}
  Consider the bounded linear form
  \begin{align}
   \ell(\boldsymbol\varphi)\coloneqq\int_{\Omega}\skalarProd{\boldsymbol u}{\boldsymbol\varphi} \,\mathrm{d} \boldsymbol x
   ,
  \end{align}
  and the bounded and symmetric bilinear form
  \begin{align}
   B(\boldsymbol v,\boldsymbol \varphi)\coloneqq\int_{\Omega} \mathscr{A}[\D \boldsymbol v]\boldsymbol{:} \mathscr{A}[\D\boldsymbol \varphi]\,\mathrm{d}\boldsymbol x
   ,
  \end{align}
 for $\boldsymbol v,\boldsymbol \varphi\in\hil^1_0(\Omega;\R^d)$. By the ellipticity of $\mathscr{A}[\boldsymbol{\cdot}\otimes\boldsymbol{\nabla}]$, we can apply \cref{lem:Aronszajn} to conclude the coercivity of $B$. Using Lax--Milgram, we then find a unique solution $\boldsymbol v\in\hil^1_0(\Omega;\R^d)$, such that
 \begin{align}\label{eq:LM}
   B(\boldsymbol v,\boldsymbol \varphi) = \ell(\boldsymbol\varphi)\quad \forall\, \boldsymbol\varphi \in\hil^1_0(\Omega;\R^d).
 \end{align}
 Thus, in particular, for all $\boldsymbol\varphi \in\hil^1_0(\Omega;\R^d)$, we have that
 \begin{equation}
  \begin{aligned}
  \int_{\Omega}\skalarProd{\boldsymbol u}{\boldsymbol\varphi} \,\mathrm{d} \boldsymbol x &=\int_{\Omega} \mathscr{A}[\D \boldsymbol v]\boldsymbol{:} \mathscr{A}[\D\boldsymbol \varphi]\,\mathrm{d}\boldsymbol x \overset{\mathclap{\mathscr{A} \text{o.p.}}}{=} \int_{\Omega} \mathscr{A}[\D \boldsymbol v]\boldsymbol{:} \D\boldsymbol \varphi\,\mathrm{d}\boldsymbol x \\
  &\overset{\mathclap{\boldsymbol \varphi\in\hil^1_0}}{=} -\int_{\Omega}\skalarProd{\Div \mathscr{A}[\D \boldsymbol v]}{\boldsymbol\varphi}\,\mathrm{d}\boldsymbol x,
  \end{aligned}
 \end{equation}
 meaning that $\boldsymbol v\in\hil^1_0(\Omega;\R^d)$ is a weak solution of
 \begin{align}\label{eq:bvpv}
  \left\{
  \begin{alignedat}{2}
   -\Div \mathscr{A}[\D \boldsymbol v] &= \boldsymbol u \quad &&\text{in } \Omega,\\
   \hphantom{ -\Div \mathscr{A}[\D ] }\boldsymbol v &= 0 \quad &&\text{on } \partial\Omega.
  \end{alignedat}
  \right.
 \end{align}
 Since $\mathscr{A}[\boldsymbol{\cdot}\otimes\boldsymbol{\nabla}]$ is elliptic, we can apply classical elliptic regularity theory (cf.~\cite[Thm.~10.3, 3.18]{giustiDirectMethodsCalculus2005} with the Legendre--Hadamard condition from \cref{rem:LH} and also the literature mentioned after \cite[Prop. 2.11]{gmeinederPartialRegularityBV2019}) to conclude that for the unique solution, we moreover have
 \begin{align}\label{eq:elliptreg}
  \boldsymbol v\in\hil^2(\Omega;\R^d)\quad \text{and}\quad \norm{\boldsymbol v}_{\hil^2(\Omega)}\le \tilde{c}\,(\norm{\boldsymbol u}_{\lebe^2(\Omega)}+\norm{\boldsymbol v}_{\lebe^2(\Omega)}).
 \end{align}
 Furthermore, by the Poincaré and Korn inequalities, we estimate that
 \begin{align}\label{eq:norm_v}
  \norm{\boldsymbol v}_{\lebe^2(\Omega)} \le c_1 \norm{\D\boldsymbol v}_{\lebe^2(\Omega)}\le c_2 \norm{\mathscr{A}[\D \boldsymbol v]}_{\lebe^2(\Omega)}.
 \end{align}
 Now, we show that $\boldsymbol\tau \coloneqq\mathscr{A}[\D \boldsymbol v]$ does the trick. Since $ \boldsymbol v\in\hil^2(\Omega;\R^d)\cap\hil^1_0(\Omega;\R)$ is the unique solution of \cref{eq:bvpv}, with this choice, we have $\boldsymbol\tau\in\hil^1(\Omega;\R^{d\times d}_{\mathscr{A}})$ and $\Div\boldsymbol\tau = -\boldsymbol u$. Moreover, we have that
 \begin{equation}
 \begin{aligned}
  \norm{\boldsymbol\tau}_{\lebe^2(\Omega)}^2&=\norm{\mathscr{A}[\D \boldsymbol v]}_{\lebe^2(\Omega)}^2 \overset{\mathclap{\eqref{eq:LM}}}{=}\int_{\Omega}\skalarProd{\boldsymbol u}{\boldsymbol v} \,\mathrm{d}\boldsymbol x \overset{\mathclap{\text{CBS}}}{\le}\norm{\boldsymbol u}_{\lebe^2(\Omega)}\norm{\boldsymbol v}_{\lebe^2(\Omega)}\\
  &\overset{\mathclap{\eqref{eq:norm_v}}}{\le} c_2\norm{\boldsymbol u}_{\lebe^2(\Omega)}\norm{\mathscr{A}[\D \boldsymbol v]}_{\lebe^2(\Omega)} = c_2\norm{\boldsymbol u}_{\lebe^2(\Omega)}\norm{\boldsymbol\tau}_{\lebe^2(\Omega)},
 \end{aligned}
 \end{equation}
 such that
 \begin{align}\label{eq:normtau}
   \norm{\boldsymbol\tau}_{\lebe^2(\Omega)}&\le c_2\norm{\boldsymbol u}_{\lebe^2(\Omega)}
   ,
 \end{align}
 and we can estimate the $\hil(\Div)$-norm of $\boldsymbol\tau$ as
 \begin{equation}
  \norm{\boldsymbol \tau}_{\hil(\Div)(\Omega)}=\sqrt{\norm{\boldsymbol\tau}_{\lebe^2(\Omega)}^2+\norm{\Div \boldsymbol\tau}_{\lebe^2(\Omega)}^2} \le \sqrt{c_2^2 +1 }\norm{\boldsymbol u}_{\lebe^2(\Omega)}.
\end{equation}
But we can also estimate the full $\hil^1$-norm using \cref{eq:normtau} and
\begin{equation}
\begin{aligned}\label{eq:normgradtau}
  \norm{\D\boldsymbol \tau}_{\lebe^2(\Omega)}
  &\overset{\mathclap{\mathscr{A} \text{o.p.}}}{\le}
  \norm{\D^2 \boldsymbol v}_{\lebe^2(\Omega)}
  \le
  \norm{\boldsymbol v}_{\hil^2(\Omega)} \overset{\mathclap{\eqref{eq:elliptreg}}}{\le}\tilde{c}\left(\norm{\boldsymbol u}_{\lebe^2(\Omega)}+\norm{\boldsymbol v}_{\lebe^2(\Omega)}\right)
  \\
  &
  \overset{\mathclap{\eqref{eq:norm_v}}}{\le}
  \tilde{c}\left(\norm{\boldsymbol u}_{\lebe^2(\Omega)}+ c_2 \norm{\boldsymbol \tau}_{\lebe^2(\Omega)}\right)
  \overset{\mathclap{\eqref{eq:normtau}}}{\le}
  \tilde{c}\left(1+c_2^2\right)\norm{\boldsymbol u}_{\lebe^2(\Omega)}
  .
\end{aligned}
\end{equation}
 Hence, a combination of \cref{eq:normtau} and \cref{eq:normgradtau} concludes the proof.
 \end{proof}


\section{Existence of unique weak solutions for the linear R13 equations}
We are now ready to prove the existence of a unique weak solution to the linear R13 equations. Namely, given some \((\mathcal{F},\mathcal{G}) \in \mathrm{V}'\times \mathrm{Q}'\), there exists a unique \((\boldsymbol{\mathcal{U}}, \boldsymbol{\mathcal{P}}) \in  \mathrm{V}\times\mathrm{Q}\) such that
\begin{equation}\label{eq:LR13weak}
  \left\{
  \begin{alignedat}{2}
    \mathcal{A}(\boldsymbol{\mathcal{U}},\boldsymbol{\mathcal{V}}) + \mathcal{B}(\boldsymbol{\mathcal{V}},\boldsymbol{\mathcal{P}}) &= \mathcal{F}(\boldsymbol{\mathcal{V}})
    \quad & \forall\, \boldsymbol{\mathcal{V}}{} &\in \mathrm{V},
    \\
    \mathcal{B}(\boldsymbol{\mathcal{U}},\boldsymbol{\mathcal{Q}}) &= \mathcal{G}(\boldsymbol{\mathcal{Q}})
    \quad & \forall\, \boldsymbol{\mathcal{Q}}{} &\in \mathrm{Q}.
 \end{alignedat}
  \right.
\end{equation}
To this end, we follow the classical strategy, which we briefly recall below.

\begin{theorem}[{\cite[Thm.~1.1]{brezziMixedHybridFinite1991}}]\label{thm:existence}
 Let $\mathrm{V}$ and $\mathrm{Q}$ be two Hilbert spaces, and \(a(\cdot,\cdot):\mathrm{V} \times \mathrm{V}\to\R\) and \(b(\cdot,\cdot):\mathrm{V} \times \mathrm{Q}\to\R\) be two continuous linear forms. Let us suppose that the range of the operator \(B\), associated with \(b(\cdot,\cdot)\), is closed in \(\mathrm{Q}'\); that is, there exists a constant \(k_0>0\) such that
  \begin{equation}
    \sup_{v \in \mathrm{V}} \frac{b(v,q)}{{\|v\|}_\mathrm{V}} \ge k_0 {\|q\|}_{\mathrm{Q} / \ker B^\top}
    = k_0 \Big( \inf_{q_0 \in \ker B^\top} {\|q+q_0\|}_{\mathrm{Q}} \Big)
    .
  \end{equation}
  If, moreover, \(a(\cdot,\cdot)\) is invertible on \(\ker B\), that is, there exists \(\alpha_0 > 0\), such that
  \begin{align}
    \inf_{u_0 \in \ker B} \sup_{v_0 \in \ker B} \frac{a(u_0,v_0)}{{\|u_0\|}_\mathrm{V} {\|v_0\|}_\mathrm{V}} \ge \alpha_0
    \quad \text{and}\quad
    \inf_{v_0 \in \ker B} \sup_{u_0 \in \ker B} \frac{a(u_0,v_0)}{{\|u_0\|}_\mathrm{V} {\|v_0\|}_\mathrm{V}} \ge \alpha_0
    ,
  \end{align}
  then there exists a solution to
\begin{equation}
  \left\{
    \begin{alignedat}{2}
      a(u,v) + b(v,p) &= f(v)
      \quad & \forall\, v &\in \mathrm{V},
      \\
      b(u,q) &= g(q)
      \quad & \forall\, q &\in \mathrm{Q},
   \end{alignedat}
    \right.
  \end{equation}
  for any \(f \in \mathrm{V}'\) and any \(g \in \operatorname{im} B\). The first component \(u\) is unique, and \(p\) is defined up to an element of \(\ker B^\top\). Moreover, one has the bounds
  \begin{align}
    \|u\|_{\mathrm{V}}
    &\le
    \frac{1}{\alpha_0} \|f\|_{\mathrm{V}'} + \left(\frac{\|a\|}{\alpha_0} + 1\right) \frac{1}{k_0} \|g\|_{\mathrm{Q}'}
    \\
    \| p \|_{\mathrm{Q} / \ker B^\top}
    &\le
    \frac{1}{k_0} \left( 1 + \frac{\| a \|}{\alpha_0} \right) \| f \|_{\mathrm{V}'} + \frac{\| a \|}{k_0^2} \left( 1 + \frac{\| a \|}{\alpha_0} \right) \| g \|_{\mathrm{Q}'}.
  \end{align}
\end{theorem}

\subsection{Coercivity on the kernel}\label{sec:CoerKern}
To continue proving the coercivity of \(\mathcal{A}\) on the kernel of $B$ (i.e., the linear map associated with \(\mathcal{B}\); see \cref{s_mixedform}), we first need the following characterization.
\begin{lemma}[Kernel of \(B\)]\label{lem:kerB}
  We have
  \begin{equation}
    \ker B = \left\{ (\boldsymbol{\sigma},\boldsymbol{s},p) \in \V : \Div\boldsymbol{\sigma} = - \boldsymbol{\nabla} p , \div\boldsymbol{s}=0 \right\}
    .
  \end{equation}
\end{lemma}
\begin{proof}
  By the definition of the kernel, we have for all \(\boldsymbol{\mathcal{U}} \in \ker B\):
  \begin{alignat}{2}
    \mathcal{B}(\boldsymbol{\mathcal{U}},\boldsymbol{\mathcal{Q}})
    &=
    0
    \quad& \forall\, \boldsymbol{\mathcal{Q}} &\in \mathrm{Q}
    \\
    \Leftrightarrow
    -
    e(\boldsymbol{v},\boldsymbol{\sigma})
    -
    g(p,\boldsymbol{v})
    -
    b(\kappa, \boldsymbol{s})
    &=
    0
    \quad& \forall\, (\boldsymbol{v},\kappa) &\in \mathrm{Q}
    .
  \end{alignat}
 Since this holds for any \((\boldsymbol{v},\kappa) \in \mathrm{Q}\), we extract two separated conditions, i.e., \(e(\boldsymbol{v},\boldsymbol{\sigma}) + g(p,\boldsymbol{v}) = 0\) and \(b(\kappa, \boldsymbol{s}) = 0\). We insert the explicit expressions of these bilinear forms from \cref{s_mixedform} and directly conclude that \(\Div\boldsymbol{\sigma} = - \boldsymbol{\nabla} p\) and \(\div\boldsymbol{s}=0\) define the kernel of \(B\).
\end{proof}

\begin{lemma}[Coercivity of $\mathcal{A}$ on $\ker B$]\label{lem:coercivityAonKernel}
The bilinear form $\mathcal{A}$ is coercive on $\ker B$, i.e., there exists a constant \(\alpha_0 > 0\) such that \(\mathcal{A}(\boldsymbol{\mathcal{U}},\boldsymbol{\mathcal{U}}) \ge \alpha_0 \norm{\boldsymbol{\mathcal{U}}}_\V^2\) for all \(\boldsymbol{\mathcal{U}} \in \ker B\).
\end{lemma}
\begin{proof}
 By \cref{lem:kerB}, with $\boldsymbol{\mathcal{U}} = (\boldsymbol{\sigma},\boldsymbol{s},p) \in \ker B$, we have $\Div\boldsymbol\sigma=-\boldsymbol{\nabla} p$. Thus, we estimate that
 \begin{align}\label{eq:fromKerB}
 \norm{\boldsymbol\sigma}_{\hil^1(\Omega)}^2
 &\ge
 \norm{\D\boldsymbol\sigma}_{\lebe^2(\Omega)}^2
 \overset{\mathclap{(\ast)}}{\ge}
 \norm{\Div\boldsymbol\sigma}_{\lebe^2(\Omega)}^2
 =
 \norm{\boldsymbol{\nabla} p}_{\lebe^2(\Omega)}^2
 ,
 \end{align}
 using, in the lower bound (\(\ast\)), the CBS inequality for the Frobenius inner products in \(\int_\Omega \|\Div \boldsymbol{\sigma}\|_{\mathbb{R}^3}^2 \,\mathrm{d} \boldsymbol{x} = \int_\Omega \| [\langle (\delta_{kl})_{k,l=1}^3 , (\partial_k \sigma_{il})_{k,l=1}^3 \rangle_{\textnormal{F}}]_{i=1}^3\|_{\mathbb{R}^3}^2 \,\mathrm{d} \boldsymbol{x} \le \|\bbone_3\|_{\textnormal{F}}^2 \int_\Omega \|\D\boldsymbol{\sigma}\|_{\textnormal{F}}^2 \,\mathrm{d} \boldsymbol{x}\). Then, we conclude with  
 \begin{equation}
 \begin{aligned}\label{eq:coercivityA}
  \mathcal{A}(\boldsymbol{\mathcal{U}},\boldsymbol{\mathcal{U}})&=a(\boldsymbol s,\boldsymbol s)
  +
  \bar{d}((\boldsymbol\sigma,p),(\boldsymbol\sigma,p))
  \\
  &\overset{\mathclap{\eqref{eq:forA},\eqref{eq:forD}}}{\ge}
  c_2\,\left(\norm{\boldsymbol s}_{\hil^1(\Omega)}^2+\norm{\boldsymbol\sigma}_{\hil^1(\Omega)}^2\right) + \tilde{\chi}\epsilon^{\textnormal{w}} \|p\|^2_{\mathrm{L}^2(\Gamma)} \\
  &\overset{\mathclap{\eqref{eq:fromKerB}}\phantom{1}}{\ge}
  \frac{c_2}{2}
  \left(\norm{\boldsymbol s}_{\hil^1(\Omega)}^2+\norm{\boldsymbol\sigma}_{\hil^1(\Omega)}^2\right)
  +
  \frac{c_2}{2}\norm{\boldsymbol{\nabla} p}_{\lebe^2(\Omega)}^2
  +
  \tilde{\chi}\epsilon^{\textnormal{w}} \|p\|^2_{\mathrm{L}^2(\Gamma)}\\
  &\overset{\mathclap{(\ast)}}{\ge}
  \alpha_0\left(\norm{\boldsymbol\sigma}_{\hil^1(\Omega)}^2+\norm{\boldsymbol s}_{\hil^1(\Omega)}^2+\norm{p}_{\hil^1(\Omega)}^2\right)
  =
  \alpha_0\norm{\boldsymbol{\mathcal{U}}}_\V^2 \quad \forall\, \boldsymbol{\mathcal{U}} \in \ker B
  ,
 \end{aligned}
 \end{equation}
 where the step (\(\ast\)) in \cref{eq:coercivityA}, depending on the value of \(\epsilon^{\textnormal{w}}\) from \cref{bcs_3d}, either uses Poincaré's inequality $\norm{p}_{\lebe^2(\Omega)}\le c\,\norm{\boldsymbol{\nabla} p}_{\lebe^2(\Omega)}$ for $p\in\hil^1(\Omega;\R)$ with $\int_{\Omega} p\,\mathrm{d}\boldsymbol x=0$ (\(\epsilon^{\textnormal{w}} = 0\)) or Friedrichs's inequality $\norm{p}_{\lebe^2(\Omega)}\le c\,(\norm{\boldsymbol{\nabla} p}_{\lebe^2(\Omega)} + \norm{p}_{\lebe^2(\Gamma)})$ for $p\in\hil^1(\Omega;\R)$ (\(\epsilon^{\textnormal{w}} > 0\)) in accordance with the definition of \(\mathrm{V}\) in \cref{eq_functionSpaceV}.
\end{proof}
\begin{remark}
Interestingly, for the coercivity of $\mathcal{A}$ in \cref{lem:coercivityAonKernel}, we used the kernel property only to include the pressure $p$ in the estimate by relating it to the stress \(\boldsymbol{\sigma}\). It becomes clear that we must proceed in this manner when considering the case of \(\epsilon^{\textnormal{w}}=0\), in which \(\mathcal{A}\) would not contain any \(p\)-terms at all. The property \(\operatorname{div} \boldsymbol{s} = 0\), however, is not needed for coercivity.
\end{remark}

\subsection{Sup-condition}\label{sec:supCond}
The existence of the right inverse of the matrix divergence from \cref{lem:Invdiv0}, combined with results in the vectorial case, yields the following result:
\begin{lemma}[Sup-condition of $\mathcal{B}$]\label{lem:supCond}%
 There exists a constant $k_0 > 0$ such that for all $\boldsymbol{\mathcal{Q}}\in\mathrm{Q}$, it holds
 \begin{align}
  \sup_{\boldsymbol{\mathcal{V}}\in\mathrm{V}}\frac{\mathcal{B}(\boldsymbol{\mathcal{V}},\boldsymbol{\mathcal{Q}})}{\norm{\boldsymbol{\mathcal{V}}}_{\mathrm{V}}}\ge k_0 \norm{\boldsymbol{\mathcal{Q}}}_{\mathrm{Q}}.
 \end{align}
\end{lemma}

\begin{proof}
 For $\boldsymbol{\mathcal{V}}=(\boldsymbol{\psi},\boldsymbol{r},q)\in\mathrm{V}$ and $\boldsymbol{\mathcal{Q}}=(\boldsymbol{v},\kappa)\in\mathrm{Q}$, the bilinear form \(\mathcal{B}\) reads as
 \begin{equation}
 \begin{aligned}
  \mathcal{B}(\boldsymbol{\mathcal{V}},\boldsymbol{\mathcal{Q}})&=-e(\boldsymbol{v},\boldsymbol{\psi})-g(q,\boldsymbol{v})-b(\kappa,\boldsymbol{r})\\
  &=
  -\int_{\Omega}\left(\Div\boldsymbol{\psi}\boldsymbol\cdot\boldsymbol{v}+\boldsymbol{v}\boldsymbol{\cdot}\boldsymbol{\nabla}q+\kappa\div\boldsymbol{r}\right)\,\mathrm{d} \boldsymbol{x}
  .\label{eq:binlin}
 \end{aligned}
 \end{equation}
Now, let $\boldsymbol{\mathcal{Q}}=(\boldsymbol{v},\kappa)\in\mathrm{Q}$ be given. Then, by \cref{lem:Invdiv0,lem:InvDiv1}, we find a $\boldsymbol{\tau}\in\hil^1(\Omega;\R^{3\times3}_{\stf})$ and a $\boldsymbol{t}\in\hil^1(\Omega;\R^3)$, such that
\begin{subequations}
\begin{alignat}{2}
 -\Div \boldsymbol{\tau} &= \boldsymbol{v},
 &\quad
 -\div\boldsymbol{t} &= \kappa, \label{eq:Divs}
 \\
 \norm{\boldsymbol{\tau}}_{\hil^1(\Omega)} &\le c \norm{\boldsymbol{v}}_{\lebe^2(\Omega)},
 &\quad
 \norm{\boldsymbol{t}}_{\hil^1(\Omega)} &\le c\norm{\kappa}_{\lebe^2(\Omega)},
 \label{eq:norms}
\end{alignat}
\end{subequations}
for some constant $c>0$. In particular, we choose $\tilde{\boldsymbol{\mathcal{V}}}\coloneqq(\boldsymbol{\tau},\boldsymbol{t},0)\in\mathrm{V}$ and obtain
\begin{multline}
 \sup_{\boldsymbol{\mathcal{V}}\in\mathrm{V}}\frac{\mathcal{B}(\boldsymbol{\mathcal{V}},\boldsymbol{\mathcal{Q}})}{\norm{\boldsymbol{\mathcal{V}}}_{\mathrm{V}}}
 \ge
 \frac{\mathcal{B}(\tilde{\boldsymbol{\mathcal{V}}},\boldsymbol{\mathcal{Q}})}{\|\tilde{\boldsymbol{\mathcal{V}}}\|_{\mathrm{V}}} \overset{\eqref{eq:binlin},\eqref{eq:Divs}}{=}\frac{\displaystyle\int_{\Omega}\left(\boldsymbol{v}^2+\kappa^2\right)\,\mathrm{d}\boldsymbol{x}}{(\norm{\boldsymbol{\tau}}_{\hil^1}^2+\norm{\boldsymbol{t}}_{\hil^1}^2)^{\frac12}}  
 \\
 \overset{\mathclap{\eqref{eq:norms}}}{\ge}
 \frac{\norm{\boldsymbol{v}}_{\lebe^2}^2+\norm{\kappa}_{\lebe^2}^2}{c(\norm{\boldsymbol{v}}_{\lebe^2}^2+\norm{\kappa}_{\lebe^2}^2)^{\frac12}}  
 =
 \frac{1}{c}\left(\norm{\boldsymbol{v}}_{\lebe^2}^2+\norm{\kappa}_{\lebe^2}^2\right)^{\frac12}=\frac1c\norm{\boldsymbol{\mathcal{Q}}}_{\mathrm{Q}}.  
\end{multline}
\end{proof}
\begin{remark}
  Remarkably, the choice of \(\tilde{\mathcal{V}}\) had a trivial third component in the proof of \cref{lem:supCond}, corresponding to the pressure \(p\), which is not needed for the sup-condition.
\end{remark}

\begin{corollary}
  The associated linear map $B:\mathrm{V}\to\mathrm{Q}'$ is an isomorphism.
\end{corollary}
\begin{proof}
  This property follows (see, e.g.,~\cite[Satz 3.6]{Braess}) since the corresponding bilinear form $\mathcal{B}$ is continuous, fulfills the sup-condition, and for all $\boldsymbol{\mathcal{Q}}\neq0$, we find a $\tilde{\boldsymbol{\mathcal{V}}}\in\mathrm{V}$ such that $\mathcal{B}(\tilde{\boldsymbol{\mathcal{V}}},\boldsymbol{\mathcal{Q}})\neq0$. For the latter, consider $\tilde{\boldsymbol{\mathcal{V}}}$ from the proof of \cref{lem:supCond}, for which we showed that
  \begin{equation}
  \mathcal{B}(\tilde{\boldsymbol{\mathcal{V}}},\boldsymbol{\mathcal{Q}})=\norm{\boldsymbol{\mathcal{Q}}}^2_{\mathrm{Q}}>0.
\end{equation}
\end{proof}
\begin{corollary}\label{cor:kerBtTrivial}
  For the transpose $B^\top:\mathrm{Q}\to\mathrm{V}'$, given via $\langle \boldsymbol{\mathcal{V}},B^\top\boldsymbol{\mathcal{Q}}\rangle_{\mathrm{V}\times \mathrm{Q}'}\coloneqq\mathcal{B}(\boldsymbol{\mathcal{V}},\boldsymbol{\mathcal{Q}})$, it holds
  \begin{equation}
    \ker B^\top=\{0\}.
  \end{equation}
\end{corollary}

\subsection{Proof of the main result}%
\enlargethispage{\baselineskip}  
We can now use the \cref{lem:supCond,lem:coercivityAonKernel} to obtain the main result.
\begin{theorem}[Well-posedness]\label{thm:mainR13}
  For any \((\mathcal{F},\mathcal{G}) \in \mathrm{V}'\times \mathrm{Q}'\), there exists a unique solution \((\boldsymbol{\mathcal{U}}, \boldsymbol{\mathcal{P}}) \in \mathrm{V}\times\mathrm{Q}\) to the weak linear R13 equations \cref{eq:LR13weak} that fulfills
  \begin{align}
    \|\boldsymbol{\mathcal{U}}\|_{\mathrm{V}}
    &\le
    \frac{1}{\alpha_0} \|\mathcal{F}\|_{\mathrm{V}'} + \left(\frac{\| \mathcal{A} \|}{\alpha_0} + 1\right) \frac{1}{k_0} \| \mathcal{G} \|_{\mathrm{Q}'},
    \\
    \| \boldsymbol{\mathcal{P}} \|_{\mathrm{Q}}
    &\le
    \frac{1}{k_0} \left( 1 + \frac{\| \mathcal{A} \|}{\alpha_0} \right) \| \mathcal{F} \|_{\mathrm{V}'} + \frac{\| \mathcal{A} \|}{k_0^2} \left( 1 + \frac{\| \mathcal{A} \|}{\alpha_0} \right) \| \mathcal{G} \|_{\mathrm{Q}'},
  \end{align}
  with the constants \(\alpha_0, k_0 > 0\) from \cref{lem:coercivityAonKernel,lem:supCond}.
\end{theorem}
\begin{proof}
  The existence follows directly from applying \cref{thm:existence} to our situation since $\mathcal{A}$ and $\mathcal{B}$ are both continuous (see \cref{sec:Appendix}), $\mathcal{A}$ is coercive on $\ker B$ by \cref{lem:coercivityAonKernel}, and $\mathcal{B}$ fulfills the sup-condition in \cref{lem:supCond}. Moreover, we have $\operatorname{im} B=\mathrm Q'$ and $\ker B^\top=\{0\}$ by \cref{cor:kerBtTrivial}, such that both components are unique.
\end{proof}
\begin{remark}
  The results in \cref{thm:mainR13}, which at first sight differ from the classical Poisson and Stokes theory, should in fact be compared with corresponding results for their \emph{first-order} systems \cref{eq_stokes_and_poisson_first_order} rather than with the second-order (primal) formulations. In the mixed formulation of Poisson's equation~\cite[Example~3.5]{brezziMixedHybridFinite1991} and in (pseudo)stress-velocity(-pressure) formulations of Stokes' problem, see, e.g.,~\cite{penatiNewMixedMethod2019,carstensenNumericalExperimentsArnoldWinther2012,lepeMixedMethodsVelocityPressurePseudostress2022,gaticaAnalysisVelocityPressure2010,farhloulFiniteElementAnalysis2025}, the scalar field $\theta$ and the velocity $\boldsymbol{u}$ enter only through constraint terms and therefore act as Lagrange multipliers in $\operatorname{L}^{2}(\Omega)$. In contrast, the heat flux and the (pseudo)stress are the variables whose divergence appears. The additional derivative terms in the R13 balance relations \cref{eq_balance2} yield the higher regularity $\boldsymbol{s},\boldsymbol{\sigma} \in \operatorname{H}^{1}(\Omega)$ for each fixed $\operatorname{Kn}>0$, and these terms vanish quadratically as $\operatorname{Kn}\to 0$ in \cref{eq_stokes_and_poisson_first_order}, so that we recover the classical mixed Poisson and Stokes formulations. Moreover, the R13 formulation of the momentum equation \cref{eq_balance1} explicitly contains $\boldsymbol{\nabla}p$. With $\operatorname{Div}\boldsymbol{\sigma}$ and the right-hand side in $\operatorname{L}^{2}(\Omega;\mathbb{R}^{3})$, this implies $p \in \tilde{\operatorname{H}}{}^1(\Omega)\subset\operatorname{L}^{2}(\Omega)$ and thus slightly strengthens, but does not contradict, the classical mixed Stokes theory.  
\end{remark}
\begin{remark}[Regularity]
  \cref{thm:existence} establishes the uniqueness of weak solutions with source terms \(\mathcal{F}\) and \(\mathcal{G}\) in the dual spaces \(\mathrm{V}'\) and \(\mathrm{Q}'\), respectively. In the classical theory of the Stokes system, it is a well-known result by Cattabriga~\cite{cattabrigaSuProblemaContorno1961} (see also, e.g.,~\cite{AmroucheGirault,giraultFiniteElementMethods1986}) that data with higher regularity lead to solutions with higher Sobolev regularity, provided the boundary is sufficiently smooth. Since the linearized R13 system exhibits an elliptic structure similar to the Stokes system, we anticipate that analogous regularity results hold. However, proving this formally for the R13 system is non-trivial due to the complex coupling in the tensor-valued boundary conditions and requires a separate, detailed analysis that lies beyond the scope of the present study.
\end{remark}


\section{Conclusion and future work}%

In this paper, we derived the well-posedness for the weak R13 equations. The first important step was to reformulate the problem as a grouped mixed problem within the abstract LBB framework. Due to the unique tensorial structure of the equations, which also involve matrix-valued differential equations, we derived new theoretical tools, including a tensor-valued Korn inequality for the symmetric and trace-free parts of matrix derivatives. Proving the ellipticity of such operators then led to the existence of a right inverse of the matrix-valued divergence as the second crucial ingredient needed to analyze the equations. These estimates could also be used in the context of linear elasticity if, e.g., nonstandard material laws are involved, and we presented them in a general form.
\par
Our analysis is the first step towards a numerical analysis of discretization schemes for the equations. For example, in the context of mixed finite element methods \textendash\ although promising numerical progress has been made, e.g., using generalized Taylor--Hood elements in~\cite{theisenFenicsR13TensorialMixed2021} \textendash\ no known stable element pairing is yet available. In fact, as a mixed problem, continuous results do not directly transfer to the discrete case. Especially the discrete sup-condition and the handling of the discrete kernel of \(B\) are unclear for now, since the symmetry and trace-free constraints on the tensor \(\boldsymbol{\sigma}\) are challenging; see, e.g., the discussion in~\cite[Sec.\ 9.1]{boffiMixedFiniteElement2013}.
\par
Further aspects include the extension to the time-dependent case, the construction of efficient preconditioned iterative solvers, and, eventually, the solution and analysis of the complete nonlinear system. On the theoretical side, reformulating the system as a variational optimization problem, conducting alternative analyses using different groupings, using augmented formulations, and extending to higher-order moment systems are natural next steps. Such systems will require coercivity estimates for symmetric and trace-free tensor derivatives of even higher order, as well as suitable right inverses of the corresponding divergence operators.


\appendix\crefalias{section}{appendix}
\section{Continuity of the bilinear forms}\label{sec:Appendix}%
The continuity of the bilinear form $\mathcal{B}$ is straightforward. For any $\boldsymbol{\mathcal{U}}=(\boldsymbol{\sigma},\boldsymbol{s},p) \in \mathrm{V}$ and any $ \boldsymbol{\mathcal{Q}}=(\boldsymbol{v},\kappa) \in \mathrm{Q}$, we have
\begingroup{}%
\allowdisplaybreaks[1]%
\begin{align}
  |\mathcal{B}(\boldsymbol{\mathcal{U}},\boldsymbol{\mathcal{Q}})|
  &=
  |\mathcal{B}({(\boldsymbol{\sigma},\boldsymbol{s},p)},{(\boldsymbol{v},\kappa)})|
  \overset{\mathclap{\triangle\textnormal{-in.}}}{\le}
  |e(\boldsymbol{v},\boldsymbol{\sigma})|
  +
  |g(p,\boldsymbol{v})|
  +
  |b(\kappa, \boldsymbol{s})|
  \\ \tag*{}
  &=
  \left|\int_\Omega \Div\boldsymbol{\sigma} \boldsymbol{\cdot} \boldsymbol{v} \,\textnormal{d}\boldsymbol{x}\right|
  +
  \left|\int_\Omega \boldsymbol{v} \boldsymbol{\cdot} \boldsymbol{\nabla} p \,\textnormal{d}\boldsymbol{x}\right|
  +
  \left|\int_\Omega \kappa \, \div\boldsymbol{s} \,\textnormal{d}\boldsymbol{x}\right|
  \\ \tag*{}
  &\overset{\mathclap{\textnormal{CBS}}}{\le}
  {\|\Div\boldsymbol{\sigma}\|}_{\lebe^2(\Omega)} {\|\boldsymbol{v}\|}_{\lebe^2(\Omega)}
  +
  {\|\boldsymbol{v}\|}_{\lebe^2(\Omega)} {\|\boldsymbol{\nabla} p\|}_{\lebe^2(\Omega)}
  +
  {\|\kappa\|}_{\lebe^2(\Omega)} {\|\div\boldsymbol{s}\|}_{\lebe^2(\Omega)}
  \\ \tag*{}
  &\le
  {\|\boldsymbol{\sigma}\|}_{\hil^1(\Omega)} {\|\boldsymbol{v}\|}_{\lebe^2(\Omega)}
  +
  {\|\boldsymbol{v}\|}_{\lebe^2(\Omega)} {\|p\|}_{\hil^1(\Omega)}
  +
  {\|\kappa\|}_{\lebe^2(\Omega)} {\|\boldsymbol{s}\|}_{\hil^1(\Omega)}
  \\ \tag*{}
  &\le
  C {\|\boldsymbol{\mathcal{U}}\|}_{\mathrm{V}} {\|\boldsymbol{\mathcal{Q}}\|}_{\mathrm{Q}}
  ,
\end{align}
\endgroup%
where we used the norm equivalence on \(\mathbb{R}^n\) in the last step. The continuity of $\mathcal{A}$ follows similarly. We demonstrate this here, as an example, for the bilinear form \(a\) as
\begingroup{}%
\allowdisplaybreaks[1]%
\begin{align}
  \abs{a(\boldsymbol{s},\boldsymbol{r})}
  &\le
  c\,\big\{{\|\sym\D \boldsymbol{s}\|}_{\lebe^2(\Omega)}\norm{\sym \D \boldsymbol{r}}_{\lebe^2(\Omega)}+ \norm{\div \boldsymbol{s}}_{\lebe^2(\Omega)}\norm{\div \boldsymbol{r}}_{\lebe^2(\Omega)} 
  \\ \tag*{}
  & \hspace{1.0em} +\norm{\boldsymbol{s}}_{\lebe^2(\Omega)}\norm{\boldsymbol{r}}_{\lebe^2(\Omega)} + \norm{s_n}_{\lebe^2(\Gamma)}\norm{r_n}_{\lebe^2(\Gamma)}+ \textstyle{\sum_{i=1}^2} \norm{s_{t_i}}_{\lebe^2(\Gamma)}\norm{r_{t_i}}_{\lebe^2(\Gamma)}\big\} 
  \\ \tag*{}
  & \le
  c_1\,\big\{{\|\D\boldsymbol{s}\|}_{\lebe^2(\Omega)}\norm{\D\boldsymbol{r}}_{\lebe^2(\Omega)}+\norm{\boldsymbol{s}}_{\lebe^2(\Omega)}\norm{\boldsymbol{r}}_{\lebe^2(\Omega)}+\norm{\boldsymbol{s}}_{\lebe^2(\Gamma)}\norm{\boldsymbol{r}}_{\lebe^2(\Gamma)}\big\}
  \\ \tag*{}
  &\le
  c_2\norm{\boldsymbol{s}}_{\hil^1(\Omega)}\norm{\boldsymbol{r}}_{\hil^1(\Omega)}
  ,
\end{align}
\endgroup%
where the last step used the trace theorem~\cite[p.~42]{Braess} to estimate the $\lebe^2$-norm on \(\Gamma\).

\section*{Acknowledgments}%
LT would like to thank Matthias Kirchhart for many fruitful discussions on mixed finite element methods back in 2022. MT thanks Jan Giesselmann and Tabea Tscherpel for valuable discussions. PL is thankful for the kind invitation to the RWTH Aachen, where parts of the project were concluded. We also thank the reviewers for their valuable comments and suggestions, which have greatly improved the paper.

\apptocmd{\sloppy}{\vbadness10000\relax}{}{}
\apptocmd{\sloppy}{\hbadness10000\relax}{}{}
\bibliographystyle{siamplain}
\bibliography{refs}

\end{document}